\documentclass[a4paper,10pt]{amsart}
\usepackage[utf8]{inputenc}
\usepackage{amsthm}
\usepackage{amsmath}
\usepackage{bm}
\usepackage{amsfonts}
\usepackage{amssymb}
\usepackage{mathtools}
\usepackage{mathrsfs}
\usepackage{setspace}
\usepackage{xcolor}
\usepackage{textgreek}
\usepackage{todonotes}
\usepackage{thmtools} 

\usepackage[pdfdisplaydoctitle,colorlinks,breaklinks,urlcolor=blue,linkcolor=blue,citecolor=blue]{hyperref} 

\newtheorem{thm}{Theorem}[section]

\newtheorem{definition}[thm]{Definition}
\newtheorem{cor}[thm]{Corollary}

\newtheorem{lem}[thm]{Lemma}

\newtheorem{prop}[thm]{Proposition}

\theoremstyle{remark}
\newtheorem{rmk}{Remark}

\newcommand{\R}{\mathbb{R}}

\newcommand{\N}{\mathbb{N}}

\title[2D Euler equations stochastic model reduction]{2D Euler equations with Stratonovich transport noise as a large scale
stochastic model reduction}
\author[F. Flandoli]{Franco Flandoli}
  \address{Scuola Normale Superiore, Piazza dei Cavalieri, 7, 56126 Pisa, Italia}
  \email{\href{mailto:franco.flandoli@sns.it}{franco.flandoli@sns.it}}
\author[U. Pappalettera]{Umberto Pappalettera}
  \address{Scuola Normale Superiore, Piazza dei Cavalieri, 7, 56126 Pisa, Italia}
  \email{\href{mailto:umberto.pappalettera@sns.it}{umberto.pappalettera@sns.it}}
\date\today

\begin{document}

\begin{abstract}
The limit from an Euler type system to the 2D Euler equations with
Stratonovich transport noise is investigated.
A weak convergence result for the vorticity field and a strong convergence result for the velocity field are proved. Our results aim to provide a  stochastic reduction of fluid-dynamics models with three different time scales.
\end{abstract}

\maketitle


\section{Introduction}

This work deals with the 2D Euler equations in vorticity form on the two
dimensional torus $\mathbb{T}^{2}=\mathbb{R}^{2}/\mathbb{Z}^{2}:$
\begin{align} \label{eq:xi_intro}
	\begin{cases}
\partial_{t}\xi+u\cdot\nabla\xi  =0, \\
\xi|_{t=0}  =\xi_{0}, 
	\end{cases}
\end{align}
where $\xi : \mathbb{T}^2 \to \R$ is the vorticity field and 
\begin{align*}
u  =K\ast\xi, \quad
\operatorname{div}u  =0,
\end{align*}
is the solenoidal velocity vector field reconstructed from $\xi$ using the Biot-Savart kernel $K$:
\begin{align*}
K\ast\xi = -\nabla^\perp (-\Delta)^{-1} \xi.
\end{align*}

Simulations of this ideal model, as well as observations of roughly two dimensional physical
systems like certain layers of the atmosphere, show a superposition of vortex
structures of different size. The basic idea behind this work is that, with a
great degree of approximation, one could describe the motion of large scale
structures by a stochastic version of 2D Euler equations, where the noise
replaces part of the influence of small scale structures on large scale ones. This
fits with the general idea of stochastic model reduction \cite{MaTiVE01,FrMaVa05,FrMa06,JaTiVa15,FrOlRaBa19}, but the precise
formulation given here is new to our knowledge.

Mathematically speaking, we present a convergence result from a system of two,
coupled, Euler-type equations to a single stochastic Euler equation with
transport type Stratonovich noise. Behind the theoretical statement, there is
a heuristic motivation based on three time scales, carefully described in
\autoref{sec:mot}.

Let us start with the mathematical result. The system of two,
coupled, Euler-type equations we consider is this work is the following:%
\begin{align} \label{eq:full_euler}\tag{E}
	\begin{cases}
d\xi_{\text{L}}^{\epsilon}+u_{\text{L}}^{\epsilon}\cdot\nabla\xi_{\text{L} %
}^{\epsilon} dt  = 
-u_{\text{S}}^{\epsilon}\cdot\nabla\xi_{\text{L} %
}^{\epsilon} dt,\\
d\xi_{\text{S}}^{\epsilon}+u_{\text{L}}^{\epsilon}\cdot\nabla\xi_{\text{S}%
}^{\epsilon}dt   =-\epsilon^{-2}%
\xi_{\text{S}}^{\epsilon}dt+\epsilon^{-2}dW,\\
u_{\text{L}}^{\epsilon}=K\ast\xi_{\text{L}}^{\epsilon},\\
u_{\text{S}}^{\epsilon}=K\ast\xi_{\text{S}}^{\epsilon},\\
\xi_{\text{L}}^{\epsilon}|_{t=0}   =\xi_{0}
,\quad
\xi_{\text{S}}^{\epsilon
}|_{t=0}=\xi_{\text{S}}^{0,\epsilon}.
	\end{cases}
\end{align}
Here $\epsilon>0$ is a scaling parameter and $\left(  W_{t}\right)  _{t\geq0}$
is a space-dependent Brownian motion of the form
\begin{align} \label{eq:W_intro}
W_{t}(x)=\sum_{k\in\mathbb{N}}\theta_{k}(x)\beta_{t}^{k},
\end{align}
where the family $\{\beta^{k}\}_{k\in\mathbb{N}}$ is made of independent
standard Brownian motions and the coefficients $\theta_{k}$ are solenoidal,
periodic and zero mean, sufficiently regular and decrease sufficiently fast
with respect to $k$, in a suitable sense to be determined later. 

The subscripts in the two components $\left(  \xi_{\text{L}}^{\epsilon} ,\xi_{\text{S}}^{\epsilon}\right)  $ refer to large scales and
small scales. 
{ For the sake of simplicity, we take the initial condition $\xi_{\text{S}}^{0,\epsilon}$ to be distributed as the invariant measure of the linear part of the equation for $\xi_{\text{S}}^{\epsilon}$ (see \autoref{sec:not} for details), but a more general initial condition in the small scale dynamics can be easily handled. }

Our main result is the following, the precise
meaning of solution to \eqref{eq:full_euler} being given by \autoref{prop:well_posedness} below:
\begin{thm}
\label{thm:intro} 
Let $T>0$ and suppose we are given a zero-mean $\xi_0 \in L^\infty(\mathbb{T}^2)$.
Denote $B_{t}(x)=-K\ast W_{t}(x)$ and let
$\xi_{\text{L}}$ be the unique solution of the stochastic equation%
\begin{align} \label{eq:xi_lim_intro}
	\begin{cases}
d\xi_{\text{L}}+u_{\text{L}}\cdot\nabla\xi_{\text{L}}dt  = 
\nabla\xi_{\text{L}}\circ dB,\\
u_{\text{L}}  =K\ast\xi_{\text{L}},\\
\xi_{\text{L}}|_{t=0}  =\xi_{0}.
	\end{cases}
\end{align}
Then, under suitable assumptions on the coefficients $\theta_{k}$, the process
$\xi_{\text{L}}^{\epsilon}$ solution of \eqref{eq:full_euler} converges as
$\epsilon\rightarrow0$ to $\xi_{\text{L}}$ in the following sense: for every
$f\in L^{1}(\mathbb{T}^{2})$:
\[
\mathbb{E}\left[  \left\vert \int_{\mathbb{T}^{2}}\xi_{\text{L}}^{\epsilon
}(t,x)f\left(  x\right)  dx-\int_{\mathbb{T}^{2}}\xi_{\text{L}}(t,x)f\left(
x\right)  dx\right\vert \right]  \rightarrow0
\]
as $\epsilon\rightarrow0$, for every fixed $t\in\lbrack0,T]$ and in
$L^{p}([0,T])$ for every finite $p$. Under the same assumptions on the
coefficients $\theta_{k}$, the velocity field $u_{\text{L}}^{\epsilon}=K\ast\xi_{\text{L}%
}^{\epsilon}$ converges as $\epsilon\rightarrow0$, in mean value, to
$u_{\text{L}}=K\ast\xi_{\text{L}}$, as variables in $C([0,T],L^{1}%
(\mathbb{T}^{2},\mathbb{R}^{2}))$.
\end{thm}

Equations of fluid mechanics with Stratonovich transport noise like \eqref{eq:xi_lim_intro} received great attention in recent years. Precursors already
appeared several years ago, see for instance \cite{BrCaFl91,BrCaFl92,MiRo04,MiRo05}. Then it was observed, for particular models (see for
instance \cite{FlGuPr10,Ma11,FlMaNe14,
BaBeFe14,Bi13,BeFlGuMa14,Fl11,BiFl20} and others) that such noise has sometimes
rich regularizing properties, typically in terms of improved uniqueness
results or blow-up control. This also contributed to additional investigations
on such random perturbation. More recently, the problem of which precise Stratonovich
transport-advection noise should be considered was understood by \cite{Ho15}
by the development of a stochastic geometric approach based on a variational
principle; concerning this important issue, let us mention that the correct
noise term for the vorticity equation in 3-dimensions has the form $\nabla
\xi_{\text{L}}\circ dB-\xi_{\text{L}}\circ d\nabla B$, which reduces in 2D to
$\nabla\xi_{\text{L}}\circ dB$, the noise used in the theorem above (see for
instance \cite{CrFlHo19} for a rigorous result in the 3D case). Our result here,
therefore, adds further motivation for the use of this kind of random
perturbations; see also \cite{CoGoHo17,GaHo18} for a justification of this
noise from a viewpoint that has certain conceptual similarities with our one here.

In \autoref{sec:mot} we describe in detail why a system for
$\left(  \xi_{\text{L}}^{\epsilon},\xi_{\text{S}}^{\epsilon}\right)$ like \eqref{eq:full_euler} above may
arise in applications. It is not only a question of splitting the global
vorticity field in two parts; a central detail, responsible for the final
result, is the precise scaling $\epsilon^{-2}\xi_{\text{S}}^{\epsilon
}dt+\epsilon^{-2}dW$. It is not obvious, a priori, why this scaling should
appear, since the usual stochastic equations with a scaling parameter that
appear in the literature have the form $\epsilon^{-2}\xi_{\text{S}}^{\epsilon
}dt+\epsilon^{-1}dW$. But when there are three time scales in the system, with
the features outlined in \autoref{sec:mot}, the special scaling of our model is
natural. See \cite{MaTiVE01} and subsequent works for
similar arguments, that were the basis of our research, although other
aspects are basically different - in particular the finite dimensionality of
the limit models in those works. As remarked in \autoref{sec:mot}, one issue over
others is critical in the approximations: the inverse cascade is not properly
captured by this model. This is however a general open problem in the realm of
stochastic model reduction.

Our final result looks like a particular issue of the general Wong-Zakai
approximation principle \cite{WoZa65}. For the Euler system, it seems the first result in
this direction. In the case of Navier-Stokes equations, other forms are
already known, see \cite{HoLeNi19,HoLeNi19+} based on rough path theory; for different equations, we mention among others the results contained in \cite{BrCaFl88,Tw93,TeZa06}. 

Our proof is based on a probabilistic argument for the Lagrangian dynamics associated with the problem \eqref{eq:full_euler}: in fact, the formulation itself - the meaning of solution -
adopted here is the Lagrangian one. 
For the deterministic Euler equations the
Lagrangian approach is classical, see for instance \cite{MaPu94} where it is
also used to prove existence and uniqueness of a solution of class $\xi\in
L^{\infty}\left(  \left[  0,T\right]  ,L^{\infty}\left(  \mathbb{T}%
^{2}\right)  \right)  $ for bounded measurable zero-mean initial vorticity. For the
stochastic case we rely on similar results proved in \cite{BrFlMa16}.

{
In the present work, we prove in the first place  a convergence result for the Lagrangian particle trajectories, or \emph{characteristics}. Then, relying on the measure-preserving property of characteristics, we are able to prove convergence of the vorticity fields in the sense of \autoref{thm:intro}. 
We would like to stress the following technical issue: the equation of characteristics contains the velocity field itself as drift, and a careful analysis of the Biot-Savart kernel is required to overcome this difficulty. 
We hope that our method can be generalized to other equations in dimension two similar to Euler, such as modified surface quasi-geostrophic equations \cite{ChCoWu11}.
Three dimensional models might also be included, possibly requiring a regularization of the nonlinearity as in \cite{ChHoOlTi05}.  }

The paper is organized as follows. In \autoref{sec:mot} we present the main
motivations behind this work, in particular we justify the interest in the
asymptotics as $\epsilon\rightarrow0$ of system \eqref{eq:full_euler}. In
\autoref{sec:not} we introduce a rigorous mathematical setting and give a
reformulation of the convergence $\xi_{\text{L}}^{\epsilon}\rightarrow
\xi_{\text{L}}$ in terms of the convergence of the characteristics, see below for details; here we
introduce a simplified version of system \eqref{eq:full_euler}, which is more
convenient to capture the main mathematical features of the original system
without obscuring them behind heavy calculations. Subsequent
\autoref{sec:conv_char} is devoted to the convergence of characteristics,
which relies on an argument similar to those contained in \cite{IkWa14} as well as some classical
estimates on the Biot-Savart kernel $K$. In \autoref{sec:conv_vort} we see how
the convergence of the vorticity fields (in the sense of
\autoref{thm:intro}) can be deduced from the convergence at the level of
characteristics. Finally, in \autoref{sec:full} we transpose the results
concerning the simplified system introduced in \autoref{sec:not} back to the
original system \eqref{eq:full_euler}. In the Appendix we prove the
equivalence between the Lagrangian notion of solution and the distributional
notion of solution to \eqref{eq:full_euler}, in a sense to be specified later, thus further broadening
the scope of our results.

\section{Motivations} \label{sec:mot}
In this section we discuss the motivations that justify our interest for the asymptotical behaviour as $\epsilon \to 0$ of $\xi^\epsilon_{\text{L}}$ solution to \eqref{eq:full_euler}.

First of all, we clarify from the beginning that the theory illustrated in this work applies to systems with \textit{three}
time scales, this sentence to be understood as explained below.

We need a small time scale $\mathcal{T}_{\text{S}}$ at which we observe
variations, fluctuations, of the main fields (here the vorticity field). We
need an intermediate scale $\mathcal{T}_{\text{M}}$ at which the previous
fluctuations look random, but not like a white noise, just random with a
typical time of variation of order $\mathcal{T}_{\text{S}}$ (small with
respect to $\mathcal{T}_{\text{M}}$). Then we need a third, large, time scale
$\mathcal{T}_{\text{L}}$ where, as a result of the theory, the small scale
fluctuations will appear as a white noise, of multiplicative type in the
present work. The following relation will play a role:%
\begin{equation}
\frac{\mathcal{T}_{\text{L}}}{\mathcal{T}_{\text{M}}}=\frac{\mathcal{T}%
_{\text{M}}}{\mathcal{T}_{\text{S}}}.\label{link scales}%
\end{equation}

We illustrate this framework of three time scales by means of an admittedly
phenomenological model. We think of a fluid which develops small scale
fluctuations at the time scale of $1 \,s$: think to wind, roughly two dimensional to
fit with our mathematical result, which flows over an irregular ground
producing small scale vortices and perturbations. The small scale
$\mathcal{T}_{\text{S}}$ has the order of $1 \,s$. The intermediate scale has the
order of $1 \,min$: in a minute, the fluctuations we observe appear as
random, with a typical fluctuation time of $1 \,s$. The large time scale will be
of the order of $1 \,h$: at such scale, the fluctuations will look as a white noise.

An example, always ideal, may be the atmospheric fluid over a large region,
limited to the lower layer, the one that interacts with the irregularities of
the ground (like the mountains). Not aiming to a precise description of such a
complex physical system, but just to visualize certain ideas, let us idealize
such fluid by means of 2D Euler equations with forcing, written in vorticity
form:%
\begin{align} \label{eq:euler_f}
	\begin{cases}
\partial_{t}\xi+u\cdot\nabla\xi  =f,\\
u  =K\ast\xi,
	\end{cases}
\end{align}
where $f$ represents the production of small scale perturbations by the
irregularities of hills and mountain profiles, for instance. For long run
investigations it is necessary to include other realistic terms, like a small
friction $-\alpha\xi$ and an even smaller dissipation $\nu \Delta \xi$, for some coefficients $1 \gg \alpha \gg \nu >0$, in order to dissipate
the energy introduced by $f$; but it is not essential to discuss such facts here.

\subsection{Human scale: seconds}

By human scale we mean the system observed by us, humans, who observe
distances in meters and appreciate variations over time spans of seconds. The
key quantity here is $\mathcal{T}_{\text{S}} = 1 \,s$.
Velocity $u\left(  t,x\right)  $ is measured in $m/s$ and vorticity $\xi\left(
t,x\right)  $ in $s^{-1}$.

Assume we split the initial conditions according to some reasonable rule
(geometric, spectral...), in large and small scales%
\[
\xi|_{t=0}=\xi_{\text{L}}\left(  0\right)  +\xi_{\text{S}}\left(  0\right)  .
\]
Small scales describe the wind fluctuations at space distances of $1-10$
meters;\ large scales those which impact at the regional level (national,
continental), namely with structures of size $10-1000 \,km$. We assume this
separation of scales at time $t=0$.

{
Having in mind \eqref{eq:euler_f}, and the previous splitting of the vorticity field in large and small scales, we consider the following system for the evolution of $\xi_{\text{L}}$, $\xi_{\text{S}}$:%
\begin{align} \label{eq:split}
	\begin{cases}
\partial_{t}\xi_{\text{L}}+\left(  u_{\text{L}}+u_{\text{S}}\right)
\cdot\nabla\xi_{\text{L}}   =0,\\
\partial_{t}\xi_{\text{S}}+\left(  u_{\text{L}}+u_{\text{S}}\right)
\cdot\nabla\xi_{\text{S}}  =f_{\text{S}},%
	\end{cases}
\end{align}
where $u_{\text{L}} = K \ast \xi_{\text{L}}$, $u_{\text{S}} = K \ast \xi_{\text{S}}$ and $f_{\text{S}}$ incorporates the small scale inputs due to ground
irregularities. 
We assume that $f_{\text{S}}$ includes variations at distances of $1-10$ meters, with changes in time in a range of order of $1$ second. 

It is easy to check that the splitting \eqref{eq:split} is consistent with \eqref{eq:euler_f}, in the
sense that if $(\xi_{\text{L}}, \xi_{\text{S}})$ is a solution of \eqref{eq:split}, then $\xi = \xi_{\text{L}}+\xi_{\text{S}}$ is a solution of \eqref{eq:euler_f}.
We point out, however, that \eqref{eq:split} can not be deduced from \eqref{eq:euler_f} and the separation of scales at time $t=0$, but rather it is a modelling hypothesis.
}

\subsection{Intermediate scale: minutes}

Let us observe the same system from the viewpoint of a recording device which
keeps memory of the wind, but with a time scale of minutes: $\mathcal{T}_{\text{M}} = 1 \,min$.
At such time scale, the fluctuations described in the previous subsection look random, the spatial scale beeing the same as above: $1-10$ meters.

This motivates our main modelling assumption, see also \cite{PeMa94,MaTiVE01,BoEc12}. We replace the small scales by a
stochastic equation, Gaussian conditionally to the large scales:%
\begin{align} \label{eq:model}
	\begin{cases}
\partial_{t}\xi_{\text{L}}+\left(  u_{\text{L}}+u_{\text{S}}\right)
\cdot\nabla\xi_{\text{L}}  =0,\\
\partial_{t}\xi_{\text{S}}+  u_{\text{L}} 
\cdot\nabla\xi_{\text{S}}  =-\frac{1}{\tau_{\text{M}}}\xi_{\text{S}}%
+\frac{\sigma}{\sqrt{\tau_{\text{M}}}}W_{\text{S}}^{\prime},%
	\end{cases}
\end{align}
with
\[
\tau_{\text{M}}=\frac{1}{60}%
\]
in the unit of measure of minutes.

\begin{rmk}
We cannot introduce this modelling assumption at the human scale, it is too
unrealistic. If we could, the value of the constant $\tau$ would be
$\tau_{\text{S}}=1$, in the unit of measure of seconds.
\end{rmk}

Heuristically speaking, in order to understand the phenomenology of the second
equation, let us drop the term $ u_{\text{L}}
\cdot\nabla\xi_{\text{S}}$, let us think to $W_{\text{S}}$ as a one dimensional Brownian motion, and realize
that the stochastic process defined as%
\[
\widetilde{\xi}_{\text{S}}\left(  t\right)  :=e^{-\frac{1}{\tau_{\text{M}}}%
t}\xi_{\text{S}}\left(  0\right)  +\int_{0}^{t}e^{-\frac{1}{\tau_{\text{M}}%
}\left(  t-s\right)  }\frac{\sigma}{\sqrt{\tau_{\text{M}}}}W_{\text{S}%
}^{\prime}\left(  s\right),
\]
which is a caricature of the true process $\xi_{\text{S}}$, converges very
fast (on the time-scale of minutes)\ to a stationary process, and - similarly
- it takes roughly $\tau_{\text{M}}=\frac{1}{60}$ minutes to go back to
equilibrium after a fluctuation. Thus, at the intermediate time scale $\mathcal{T}_{\text{M}}$, the
small scale process looks random, with visible variations every $\frac{1}{60}$
units of time.
Its intensity is (essentially)\ independent of $\tau_{\text{M}}$ and given by
$\sigma$: the variance of the stochastic integral in the previous formula is
$\int_{0}^{t}e^{-\frac{2}{\tau_{\text{M}}}\left(  t-s\right)  }\frac
{\sigma^{2}}{\tau_{\text{M}}}ds$. When the noise is space-dependent, the
intensity is also modulated in space, so $\sigma$ is a sort of global, mean
order of magnitude.

The replacement just discussed of the true small scale equation by a
stochastic equation has some natural motivations, discussed above, but it also
has flaws. One of them is related to the inverse cascade, which dominates the
energy transfer between scales in 2D, see \cite{Kr67}. Inverse cascade is mostly discarded in
this model, having replaced the transfer mechanism due to the term
$u_{\text{S}}\cdot\nabla\xi_{\text{S}}$ by a Gaussian term with no Fourier
exchange. We do not know how to remedy this drawback. Let us only mention
that, generally speaking, the problem of a correct energy transfer between
scales in stochastic parametrization and stochastic model reduction theories
is the most important essentially open problem; our work is not a contribution
to the solution of this extremely difficult problem but only the description
of a particular stochastic model reduction procedure, different from others
previously introduced in the literature.

\subsection{Regional scale: hours}

By this we mean the same system, lower atmospheric layer over a large region,
observed by a satellite. The unit of measure of  time is $\mathcal{T}_{\text{L}} =1 \,h$ and the unit of measure of space may be $10 - 1000 \,km$, that is now different from the spatial scale of meters proper of human and intermediate points of view.
We have chosen this scales having in mind, for instance, weather prediction.

How does it look like the system above seen at this space-time scale? If there
is no noise term, the formulae are the same as above with
\[
\tau_{\text{L}}=\frac{1}{60\times60}.
\]
But this rescaling, correct for the term $-\frac{1}{\tau_{\text{L}}}%
\xi_{\text{S}}$, does not hold true for the stochastic term $\frac{\sigma}{\sqrt{\tau}}W_{\text{S}%
}^{\prime}$. Let us see more closely the correct rescaling.

\begin{rmk}
We start to see here how the final result depends on the precise procedure
described in this section. If we had imposed the stochastic structure of
small scales from the very beginning, namely at the human level, then the
intermediate step would be unessential, since a single rescaling to the regional
level would give the same result. But the result of this alternative procedure would not be the one described in
this work, it would be different. It is essential that the passage from human
to intermediate scale is based on the rules of deterministic calculus, while
the passage from the intermediate to the regional scales is based on the rules
of stochastic calculus. Only in this way we get the scaling factors
characteristic of the theory described in this work.
\end{rmk}

In order to avoid trivial mistakes in the rescaling from intermediate to
regional scale, let us formalize in more detail the change of unit of measure.
The space and time variables at intermediate level will be denoted by $x,t$
and those at regional level by $X,T$. Essential is that the unit of measure of
$t$ is minutes and the one of $T$ is hours, differing by the factor%
\[
\epsilon^{-1}=60, \quad t = \epsilon^{-1} T.
\]
Less essential here is the role of the unit of measures of $x$ and $X$. We
assume they differ by a factor $\epsilon_{x}^{-1}$, namely 
\begin{align*}
x = \epsilon_{x}^{-1} X.
\end{align*}
The only place relevant for
applications where it will appear is in the modification of the
space-covariance of the noise, which however is not our main concern here.

Denote by $u\left(  t,x\right)  $ and $U\left(  T,X\right)  $ the velocities
in the intermediate and regional scale, respectively, and similarly by
$\xi\left(  t,x\right)  $ and $\Xi\left(  T,X\right)  $ for the vorticities.
We adopt the same notation for their large scale components $u_{\text{L}}$, $U_{\text{L}}$, $\xi_{\text{L}}$, $\Xi_{\text{L}}$ and their small scale components $u_{\text{S}}$, $U_{\text{S}}$, $\xi_{\text{S}}$, $\Xi_{\text{S}}$.
We have
\[
U\left(  T,X\right)  =\epsilon_{x}\epsilon^{-1}u\left(  \frac{T}{\epsilon
},\frac{X}{\epsilon_{x}}\right)
\]
and thus
\[
\Xi\left(  T,X\right)  =\epsilon^{-1}\xi\left(  \frac{T}{\epsilon},\frac
{X}{\epsilon_{x}}\right)  .
\]
Notice that the material derivative preserves its structure under unit measure
change, here up to the factor $\epsilon^{-2}$:
\[
\left[  \frac{\partial\Xi}{\partial T}+U\cdot\nabla_{X}\Xi\right]  \left(
T,X\right)  =\epsilon_{{}}^{-2}\left[  \frac{\partial\xi}{\partial t}%
+u\cdot\nabla_{x}\xi\right]  \left(  \frac{T}{\epsilon},\frac{X}{\epsilon_{x}%
}\right)  .
\]
Similar identities hold for \textquotedblleft mixed\textquotedblright%
\ material derivatives,  like $U_{\text{S}} 
\cdot\nabla_{X}\Xi_{\text{L}}$ and $U_{\text{L}}\cdot\nabla_{X}\Xi_{\text{S}}$.

Now, let us write the equations \eqref{eq:model} above from the viewpoint of the satellite:
\[
\left[  \partial_{T}\Xi_{\text{L}}+\left(  U_{\text{L}}+U_{\text{S}}\right)
\cdot\nabla_{X}\Xi_{\text{L}}\right]  \left(  T,X\right)  =0
\]%
\[
\left[  \partial_{T}\Xi_{\text{S}}+U_{\text{L}}\cdot\nabla_{X}\Xi_{\text{S}%
}\right]  \left(  T,X\right)  =\epsilon^{-2}\left[  -\frac{1}{\tau_{\text{M}}%
}\xi_{\text{S}}+\frac{\sigma}{\sqrt{\tau_{\text{M}}}}W_{\text{S}}^{\prime
}\right]  \left(  \frac{T}{\epsilon},\frac{X}{\epsilon_{x}}\right)  .
\]
Notice that we still have $\tau_{\text{M}}$ in these equations. Let us
elaborate the term on the right-hand-side of the second equation. First,
\begin{align*}
-\epsilon^{-2}\frac{1}{\tau_{\text{M}}}\xi_{\text{S}}\left(  \frac{T}%
{\epsilon},\frac{X}{\epsilon_{x}}\right)   &  =-\frac{1}{\tau_{\text{M}%
}\epsilon}\Xi_{\text{S}}\left(  T,X\right) \\
&  =:-\frac{1}{\tau_{\text{L}}}\Xi_{\text{S}}\left(  T,X\right)
\end{align*}
havind defined%
\[
\tau_{\text{L}}=\tau_{\text{M}}\epsilon=\frac{1}{60\times60}.
\]

Second, working with finite increments which is more clear when we deal with
Brownian motion, we have
\begin{align*}
\frac{\Delta W_{\text{S}}}{\Delta t}\left(  \frac{T}{\epsilon},\frac
{X}{\epsilon_{x}}\right)   &  =\frac{W_{\text{S}}\left(  \frac{T}{\epsilon
}+\Delta t,\frac{X}{\epsilon_{x}}\right)  -W_{\text{S}}\left(  \frac
{T}{\epsilon},\frac{X}{\epsilon_{x}}\right)  }{\Delta t}\\
&  =\epsilon\frac{W_{\text{S}}\left(  \frac{T+\epsilon\Delta t}{\epsilon
},\frac{X}{\epsilon_{x}}\right)  -W_{\text{S}}\left(  \frac{T}{\epsilon}%
,\frac{X}{\epsilon_{x}}\right)  }{\epsilon\Delta t}\\
&  \overset{\mathcal{L}}{=}\sqrt{\epsilon}\frac{\widetilde{W}_{\text{S}%
}\left(  T+\epsilon\Delta t,\frac{X}{\epsilon_{x}}\right)  -\widetilde{W}%
_{\text{S}}\left(  T,\frac{X}{\epsilon_{x}}\right)  }{\epsilon\Delta t},%
\end{align*}
for an auxiliary Brownian motion $\widetilde{W}_{\text{S}}$, namely
\[
W_{\text{S}}^{\prime}\left(  \frac{T}{\epsilon},\frac{X}{\epsilon_{x}}\right)
\overset{\mathcal{L}}{=}\sqrt{\epsilon}\widetilde{W}_{\text{S}}^{\prime
}\left(  T,\frac{X}{\epsilon_{x}}\right),
\]%
and therefore
\[
\epsilon^{-2}\frac{\sigma}{\sqrt{\tau_{\text{M}}}}W_{\text{S}}^{\prime}\left(
\frac{T}{\epsilon},\frac{X}{\epsilon_{x}}\right)  \overset{\mathcal{L}%
}{=}\epsilon^{-3/2}\frac{\sigma}{\sqrt{\tau_{\text{M}}}}\widetilde{W}%
_{\text{S}}^{\prime}\left(  T,\frac{X}{\epsilon_{x}}\right)  .
\]
Hence the equation for $U_{\text{S}}$ reads%
\[
\left[  \partial_{T}U_{\text{S}}+U_{\text{L}}\cdot\nabla_{X}U_{\text{S}%
}\right]  \left(  T,X\right)  =-\frac{1}{\tau_{\text{L}}}\Xi_{\text{S}}\left(
T,X\right)  +\frac{\sigma}{\epsilon\sqrt{\tau_{\text{L}}}}\widetilde{W}%
_{\text{S}}^{\prime}\left(  T,\frac{X}{\epsilon_{x}}\right)  .
\]
The distance at which we still may feel a correlation of the noise
$\widetilde{W}_{\text{S}}^{\prime}\left(  T,\frac{X}{\epsilon_{x}}\right)  $
is of order $\epsilon_{x}$, rescaled with respect to the intermediate level.

Recall now condition (\ref{link scales}). Translated into the new constant it
corresponds to what we have tacitly assumed, namely that $\epsilon=\tau_{\text{M}}$.
This implies $\tau_{\text{L}}=\epsilon^{2}$,
and thus%
\[
\left[  \partial_{T}U_{\text{S}}+U_{\text{L}}\cdot\nabla_{X}U_{\text{S}%
}\right]  \left(  T,X\right)  =-\frac{1}{\epsilon^{2}}\Xi_{\text{S}}\left(
T,X\right)  +\frac{\sigma}{\epsilon^{2}}\widetilde{W}_{\text{S}}^{\prime
}\left(  T,\frac{X}{\epsilon_{x}}\right),
\]
which is the form of our starting model of the rigorous theory (with different notation).

\section{Notation and preliminaries}\label{sec:not}

For any $p \in [1,\infty]$ denote $L^p_0(\mathbb{T}^2)$ the space of $p$-integrable zero-mean real functions on the two dimensional torus $\mathbb{T}^2$.
For the sake of a clear and effective presentation, we decide to study in the first place the following \textit{simplified} 2D Euler system:
\begin{equation} \label{eq:xi_eps}\tag{sE} 
\begin{cases}
d\xi^{\epsilon}_t
+u^{\epsilon}_t \cdot\nabla\xi^{\epsilon}_t dt   
=
-
\sum_{k \in \mathbb{N}} \sigma_k  \cdot\nabla\xi^{\epsilon}_t \eta^{\epsilon,k}_t dt,\\
\nonumber
u^{\epsilon}_t   =K\ast\xi^{\epsilon}_t,\\
\nonumber
\xi^{\epsilon}|_{t=0}   =\xi_{0},
\end{cases}
\end{equation}
where $\xi_{0} \in L_0^\infty(\mathbb{T}^2)$ is the (deterministic) initial condition, $K$ is the Biot-Savart kernel on the two dimensional torus $\mathbb{T}^2$, $\sigma_k = K \ast \theta_k : \mathbb{T}^2 \to \mathbb{R}^2$ and $\eta^{\epsilon,k}$ is a Ornstein-Uhlenbeck process:
\begin{equation*}
\eta^{\epsilon,k}_t = 
e^{-\epsilon^{-2} t} \eta^{\epsilon,k}_0 +
\int_0^t
\epsilon^{-2} e^{-\epsilon^{-2} (t-s)} d\beta^{k}_s,
\qquad k \in \mathbb{N}.
\end{equation*}
The family $\beta = \{\beta^{k}\}_{k \in \mathbb{N}}$ is made of independent standard Brownian motions on a given filtered probability space $(\Omega,\{\mathcal{F}_t\}_{t \geq 0},\mathbb{P})$, { and the initial conditions $\{\eta^{\epsilon,k}_0\}_{k \in \N}$ are measurable with respect to $\mathcal{F}_0$, so that the processes $\{\eta^{\epsilon,k}\}_{k \in \N}$ are progressively measurable with respect to the filtration $\{\mathcal{F}_t\}_{t \geq 0}$.
Moreover, up to a possible enlargement of the filtration $\{\mathcal{F}_t\}_{t \geq 0}$, to simplify our discussion we take independent initial conditions $\{\eta^{\epsilon,k}_0\}_{k \in \N}$, also independent of $\beta$, and distributed as centred Gaussian variables with variance equal to $\epsilon^{-2}/2$. In this way, the processes $\{\eta^{\epsilon,k}\}_{k \in \N}$ are stationary, and $\eta^{\epsilon,k}$ is independent of $\eta^{\epsilon,h}$ for $k \neq h$. 

\begin{rmk}
Notice that the process $\sum_k \sigma_k \eta^{\epsilon,k}$ is nothing but a rough approximation for the small scale vorticity $\xi_{\text{S}}^\epsilon$, obtained by simply dropping the nonlinear term in the second equation of \eqref{eq:full_euler}.
This simplified formulation of \eqref{eq:full_euler} clarifies \emph{why} we expect a Wong-Zakai result to be true for the large scale vorticity: indeed, for every $k \in \N$ the process $\eta^{\epsilon,k}$ formally converges to a white-in-time noise, because of the following computation:
\begin{align*}
\int_0^t \eta^{\epsilon,k}_s ds
&=
\int_0^t e^{-\epsilon^{-2} s} \eta^{\epsilon,k}_0 ds
+
\int_0^t \left(
\int_0^s \epsilon^{-2} e^{-\epsilon^{-2}(s-r)} d\beta^k_r \right) ds
\\
&=
\int_0^t e^{-\epsilon^{-2} s} \eta^{\epsilon,k}_0 ds
+
\int_0^t \left(
\int_r^t \epsilon^{-2} e^{-\epsilon^{-2}(s-r)} ds \right) d\beta^k_r 
\\
&=
\int_0^t e^{-\epsilon^{-2} s} \eta^{\epsilon,k}_0 ds
+
\int_0^t \left(
1 - e^{-\epsilon^{-2}(t-r)} ds \right) d\beta^k_r 
\\
&= \beta^k_t + O(\epsilon).
\end{align*}
\end{rmk}
}
 
We make the following assumption on the coefficients $\sigma_k$:
\begin{itemize}
\item[(\textbf{A1})] 
$\sigma_k \in C^2(\mathbb{T}^2,\mathbb{R}^2)$ for every $k \in \mathbb{N}$ and $\sum_{k \in \mathbb{N}} 
\| \nabla^2 \sigma_k \|_{L^\infty(\mathbb{T}^2,\mathbb{R}^8)}
< \infty $,
\end{itemize}
where $\nabla^2 \sigma_k$ is the second spatial derivative of $\sigma_k$ and is understood as a vector field taking values in $\mathbb{R}^8$:
\begin{align*}
(\nabla^2 \sigma_k(x))^\alpha_{\beta,\gamma} = \partial_{x_\beta} \partial_{x_\gamma} \sigma_k^\alpha (x), 
\quad
x \in \mathbb{T}^2, 
\quad 
\alpha,\beta,\gamma \in \{1,2\}. 
\end{align*}
Assumption (A1) above is immediately translated in the equivalent assumption on the coefficients $\theta_k$ of \eqref{eq:W_intro}:
\begin{itemize}
\item[(\textbf{A1})]
$\theta_k \in L^2_0(\mathbb{T}^2) \cap C^1(\mathbb{T}^2,\mathbb{R})$ for every $k \in \mathbb{N}$ and $\sum_{k \in \mathbb{N}} 
\| \nabla \theta_k \|_{L^\infty(\mathbb{T}^2,\mathbb{R}^2)}
< \infty $.
\end{itemize}
As an example, one can take, for $\mathbf{k} \in \mathbb{Z}^2 \setminus \{(0,0)\}$ and $e_\mathbf{k}(x) = \exp(2\pi i\mathbf{k}\cdot x)$,  
\begin{align*}
\theta_\mathbf{k} (x)
= q_\mathbf{k} e_\mathbf{k}(x), 
\qquad
q_\mathbf{k} \sim
\frac{1}{|\mathbf{k}|^{3+\delta}}, 
\mbox{ for some }\delta>0.
\end{align*}

In order to study well posedness of the system \eqref{eq:xi_eps}, we first need to specify what is the notion of solution we are going to study.
We give the following definitions:
\begin{definition}
We say that a measurable map $\varphi: \Omega \times [0,T] \times \mathbb{T}^2 \to \mathbb{T}^2$ is a \emph{stochastic flow of homeomorphisms} if:
\begin{itemize}
\item for almost every $\omega \in \Omega$, 
$\varphi(\omega,t):\mathbb{T}^2 \to \mathbb{T}^2$ is a homeomorphism for every $t \in [0,T]$;
\item for every $x \in \mathbb{T}^2$, $\varphi(x): \Omega \times [0,T] \to \mathbb{T}^2$ is progressively measurable with respect the filtration $\{\mathcal{F}_t\}_{t \in [0,T]}$. 
\end{itemize}
\end{definition}

\begin{definition}
A process $\xi \in L^\infty(\Omega \times [0,T] \times \mathbb{T}^2)$ is said to be \emph{weakly progressively measurable} if for every test function $f \in L^1(\mathbb{T}^2)$ the process 
\begin{align*}
t \mapsto \int_{\mathbb{T}^2} \xi_t(x) f(x) dx
\end{align*}
is progressively measurable with respect to the filtration $\{\mathcal{F}_t\}_{t \geq 0}$.
\end{definition}

The notion of solution to \eqref{eq:xi_eps} we adopt hereafter is the following Lagrangian formulation: 
{
\begin{definition} \label{def:sol}
Let $\epsilon>0$ and $\xi_0 \in L^\infty(\mathbb{T}^2)$. We say
that a weakly progressively measurable process $\xi^\epsilon$ is a solution to \eqref{eq:xi_eps}
if it is given by the transportation of the initial vorticity $\xi_0$ along the particle trajectories, in formulae: 
\begin{align} \label{eq:transport_eps}
\xi^{\epsilon}_t = \xi_0 \circ (\varphi^{\epsilon}_t)^{-1},
\end{align}
where $\varphi^{\epsilon}_t : \mathbb{T}^2 \to \mathbb{T}^2$ is a stochastic flow of homeomorphisms which satisfies for every $x \in \mathbb{T}^2$:
\begin{align} \label{eq:char_eps}
	\begin{cases}
d\varphi_{t}^{\epsilon}\left(  x\right) = 
u^{\epsilon}_t\left(
\varphi_{t}^{\epsilon}\left(  x\right)  \right)dt  +\sum_{k \in \mathbb{N}} \sigma_k \left( \varphi_{t}^{\epsilon}\left(  x\right)  \right) \eta^{\epsilon,k}_t dt,\\ 
\varphi_{0}^{\epsilon}\left(  x\right) = x.
	\end{cases}
\end{align}  
\end{definition}
}
We adopt the same terminology, \emph{mutatis mutandis}, for equations and systems similar to \eqref{eq:xi_eps}.

In fact, we shall prove that in this setting there exists an unique triple $(\xi^{\epsilon},u^{\epsilon},\varphi^{\epsilon})$ such that $u^\epsilon = K \ast \xi^\epsilon$, \eqref{eq:transport_eps}, and \eqref{eq:char_eps} hold simultaneously for every $t \in [0,T]$, see \autoref{prop:lagrangian} below.

The maps $\mathbb{T}^2 \ni x \mapsto \varphi_t(x)$, $t \in [0,T]$ are usually called the \emph{characteristics} associated to \eqref{eq:xi_eps}, since they describe the trajectory of an ideal fluid particle with initial position $x_0=x$. 

\begin{prop} \label{prop:lagrangian}
Assume (A1). Then, for every $\epsilon >0$ and $\xi_0 \in L^\infty(\mathbb{T}^2)$ there exists a unique stochastic flow of homeomorphisms $\varphi^\epsilon$ such that \eqref{eq:char_eps} holds with $u^{\epsilon}_t
= K \ast \xi^{\epsilon}_t $ and
\begin{equation*}
\xi^{\epsilon}_t   =\xi_{0} \circ (\varphi
_{t}^{\epsilon})^{-1}, 
\end{equation*}
as variables in $L^\infty(\mathbb{T}^2)$.
In particular, system \eqref{eq:xi_eps} is well posed in the sense of \autoref{def:sol}.

Moreover, for a.e. $\omega \in \Omega$, the map $\varphi^\epsilon_t :\mathbb{T}^2 \to \mathbb{T}^2$  is measure-preserving with respect to the Lebesgue measure on $\mathbb{T}^2$ for every $t \in [0,T]$:
\begin{align*}
\int_{\mathbb{T}^2} f(x) dx =
\int_{\mathbb{T}^2} f(\varphi^\epsilon_t(y)) dy,
\qquad
\mbox{ for every } f \in L^1(\mathbb{T}^2).
\end{align*}
\end{prop}

The proof of the previous proposition is omitted, being easily reconstructed from the proof of the analogous result for the characteristics of the full system \eqref{eq:full_euler}, see \autoref{prop:well_posedness}. Thanks to this proposition, we can finally define our notion of solution.

The presumed limit equation for $\xi^{\epsilon}$ is%
\begin{align} \label{eq:xi}
	\begin{cases}
d\xi_t+u_t\cdot\nabla\xi_t dt =
-\sum_{k \in \mathbb{N}}
\sigma_{k}\cdot\nabla\xi_t\circ d\beta^{k}_t, \\
u_t  =K\ast\xi_t,\\
\xi|_{t=0} =\xi_{0},%
	\end{cases}
\end{align}
where $\circ d\beta^{k}_t$ stands for the Stratonovich integral. 
In \cite[Section 7]{BrFlMa16} it is proved that, for $C^2$ coefficients $\sigma_k$, equation \eqref{eq:xi} above admits an unique weakly progressively measurable solution given by
\begin{equation} \label{eq:transport}
\xi_t  =\xi_{0} \circ (\varphi_{t})
^{-1},
\end{equation}
as variables in $L^\infty(\mathbb{T}^2)$, where $\varphi_{t}$ is the stochastic flow of measure-preserving homeomorphisms solution to the SDE%
\begin{align} \label{eq:char}
	\begin{cases}
d\varphi_{t}(x) =
u_t\left(  \varphi_{t}(x)\right)  dt+\sum_{k}\sigma_{k}\left(  \varphi_{t}(x)\right)
\circ d\beta_{t}^{k},\\ 
\varphi_{0}\left(  x\right) = x.
	\end{cases}
\end{align}

\subsection{Reformulation of the problem}

Recall that, since both $\varphi_{t}^{\epsilon}$ and $\varphi_{t}$ are
measure-preserving maps of the torus $\mathbb{T}^2$, for every test
function $f \in L^1(\mathbb{T}^2)$ we have the following identities:%
\begin{align*}
\int_{\mathbb{T}^2}
\xi_t^{\epsilon}(x)  f\left(  x\right)  dx
&=
\int_{\mathbb{T}^2}
\xi_{0}\left(
y\right)  f\left(  \varphi_{t}^{\epsilon}\left(  y\right)  \right)  dy,
\\
\int_{\mathbb{T}^2}
\xi_t(x)  f\left(  x\right)  dx
&=
\int_{\mathbb{T}^2}
\xi_{0}\left(  y\right)
f\left(  \varphi_{t}\left(  y\right)  \right)  dy.
\end{align*}

This motivates, in view of the meaning of convergence $\xi^{\epsilon}\rightarrow\xi$ (in a suitable sense), to investigate instead the convergence of characteristics 
$\varphi^{\epsilon}\rightarrow\varphi$,
where the characteristics $\varphi^{\epsilon}$, $\varphi$ are stochastic flows of measure-preserving homeomorphisms that solve:
\begin{align*}
d\varphi_{t}^{\epsilon}\left(  x\right)  &= 
u^{\epsilon}_t\left(
\varphi_{t}^{\epsilon}\left(  x\right)  \right)dt  +\sum_{k \in \mathbb{N}} \sigma_k \left( \varphi_{t}^{\epsilon}\left(  x\right)  \right) \eta^{\epsilon,k}_t dt,
\\
d\varphi_{t}(x) &=
u_t\left(  \varphi_{t}(x)\right)  dt+\sum_{k \in \mathbb{N}}\sigma_{k}\left(  \varphi_{t}(x)\right)
\circ d\beta_{t}^{k},
\end{align*}
keeping in mind, however, that $u^{\epsilon}$ and $u$ are not given functions,
but they depend on the other variables, in partcular they are random.
Indeed, we do not know a propri that $
u^{\epsilon}\rightarrow u$
in some sense, but this information is part of the problem (cfr. \autoref{cor:vel}).

\subsection{Properties of the Biot-Savart kernel}
Here we briefly recall some useful properties of the Biot-Savart kernel $K$. We refer to \cite{BrFlMa16} for details and proofs.

First of all, recall that the convolution $K \ast \xi$ is well-defined for every $\xi \in L^p(\mathbb{T}^2)$, $p \in [1,\infty]$ and the following estimate holds: for every $p \in [1,\infty]$ there exists a constant $C$ such that for every $\xi \in L^p(\mathbb{T}^2)$
\begin{align*}
\| K \ast \xi \|_{L^p(\mathbb{T}^2,\mathbb{R}^2)}
\leq C
\| \xi \|_{L^p(\mathbb{T}^2)}.
\end{align*}

For $p \in (1,\infty)$ and $\xi \in L^p_0(\mathbb{T}^2)$, the convolution with $K$ actually represents the Biot-Savart operator:
\begin{align*}
K \ast \xi = -\nabla^\perp (-\Delta)^{-1} \xi,
\end{align*}
which, to every $\xi \in L^p_0(\mathbb{T}^2)$, associates the unique zero-mean, divergence-free
velocity vector field $u \in W^{1,p}(\mathbb{T}^2,\mathbb{R}^2)$ such that
\begin{align*}
\mbox{curl}\,u = \xi.
\end{align*}
Moreover, for every $p \in (1,\infty)$ there exist constants $c$, $C$ such that for every $\xi \in L^p_0(\mathbb{T}^2)$
\begin{align*}
c
\| \xi \|_{L^p(\mathbb{T}^2)} 
\leq 
\| K \ast \xi \|_{W^{1,p}(\mathbb{T}^2,\mathbb{R}^2)}
\leq C
\| \xi \|_{L^p(\mathbb{T}^2)}.
\end{align*}

Let $r \geq 0$. Denote $\gamma$ the concave function:
\begin{equation*}
\gamma(r) =
r(1-\log r) \mathbf{1}_{\{0<r<1/e\}} + 
(r+1/e)		\mathbf{1}_{\{r\geq 1/e\}}.
\end{equation*}
The following two lemmas are proved in \cite{BrFlMa16}.

\begin{lem} \label{lem:log_lip}
There exists a positive constant $C$ such that:
\begin{equation*}
\int_{\mathbb{T}^2}
\left| K(x-y) - K(x'-y)\right| dy
\leq
C \gamma(|x-x'|)
\end{equation*}
for every $x,x' \in \mathbb{T}^2$.
\end{lem}

\begin{lem} \label{lem:comp}
Fix $T>0$ and let $\lambda>0$, $z_0 \in [0,\exp(1-2e^{\lambda T})]$ be constants.
Denote $z$ the unique solution of the following ODE:
\begin{align*}
z_t = z_0 + \lambda  \int_0^t \gamma(z_s) ds.
\end{align*}
Then for every $t \in [0,T]$ the following estimate holds:
\begin{align*}
z_t \leq e z_0^{\exp(-\lambda t)}.
\end{align*}
\end{lem}

Hereafter, the symbol $\lesssim$ will be used to indicate an inequality up to a multiplicative constant $C$ which depends only of the data of the problem (\emph{e.g.} $T$, $\xi_0$, $\theta_k$ etc.). However, for the sake of clarity, we always try to show in the calculations where assumption (A1) comes into play.

\section{Convergence of characteristics} \label{sec:conv_char}
For a given $y \in \mathbb{T}^2$, denote $|y|$ the geodesic distance on the flat two dimensional torus of the point $y$ from $(0,0) \in \mathbb{T}^2$.
To keep the notation as simple as possible, we define, for a measurable map $\varphi$ from $\mathbb{T}^2 $ to itself, the following quantity:
\begin{align*}
\left\| \varphi \right\|_{L^1(\mathbb{T}^2,\mathbb{T}^2)}
=
\int_{\mathbb{T}^2}
\left| \varphi(x) \right|dx.
\end{align*}
We adopt this notation because of the similarity with the norm of the Banach space $L^1(\mathbb{T}^2,\mathbb{R}^2)$, altough $\left\| \cdot \right\|_{L^1(\mathbb{T}^2,\mathbb{T}^2)}$ is not a norm on the space of measurable maps $\mathbb{T}^2 \to \mathbb{T}^2$, in particular it is not positively homogeneus.
In a similar fashion, we define $\left\| \cdot \right\|_{L^\infty(\mathbb{T}^2,\mathbb{T}^2)}$ as
\begin{align*}
\left\| \varphi \right\|_{L^\infty(\mathbb{T}^2,\mathbb{T}^2)}
=
\mbox{ess}\sup_{x \in \mathbb{T}^2}
\left| \varphi(x) \right|.
\end{align*}

In this section we prove the following result,
concerning convergence of the characteristics $\varphi^\epsilon$ of the simplified system \eqref{eq:xi_eps} towards the characteristics $\varphi$ of system \eqref{eq:xi}.
 
\begin{prop} \label{prop:char}
Assume (A1). Let $\varphi^\epsilon$
be the solution of \eqref{eq:char_eps} and let $\varphi$
be the solution of \eqref{eq:char}.
Then, for every $T>0$, the following convergence holds as $\epsilon \to 0$:
\begin{align} \label{eq:conv_char}
\mathbb{E} \left[ \sup_{s \leq T} \left\| 
\varphi^\epsilon_{s}-\varphi_s \right\|_{L^1(\mathbb{T}^2,\mathbb{T}^2)}\right] \to 0.
\end{align}
\end{prop}

The strategy of the proof is the following, and is taken from \cite{IkWa14}{,  see also \cite{AsFlPa20+}}.
The idea is to discretize the time interval $[0,T]$ into subintervals of the form $[n\delta_\epsilon,(n+1)\delta_\epsilon]$, for a suitable choice of the mesh $\delta_\epsilon$.
Then, we adapt an argument in \cite{IkWa14} that gives a control of the noisy part of the equations for the characteristics $\varphi^\epsilon$, $\varphi$ in the regime
$\delta_\epsilon^2/\epsilon^3 \to 0$, $\delta_\epsilon/\epsilon^2 \to \infty$.
The nonlinear drift is controlled by \autoref{lem:log_lip}. Finally, \autoref{lem:comp} gives the convergence \eqref{eq:conv_char}.

\subsection{Estimates on the increments}
In this paragraph we give some preliminary estimates on the increments of the characteristics $\varphi^\epsilon$, $\varphi$.
We will make use of the following lemma on the supremum of the Ornstein-Uhlenbeck process, which can be found in \cite{JiZh20}.

\begin{lem} \label{lem:OU}
Let $T>0$, $p\geq 1$. Then for every $k \in \mathbb{N}$: 
\begin{align*}
\mathbb{E} \left[ 
\sup_{s\in[0,T]} |\eta^{\epsilon,k}_s|^p
\right]
\lesssim
\epsilon^{-p} 
{ \log^{p/2}(1+\epsilon^{-2})}.
\end{align*}
\end{lem}

{ Hereafter, we absorb every factor $\log(1+\epsilon^{-2})$ coming from \autoref{lem:OU} in the symbol $\lesssim$. Since we are only interested in the limit $\epsilon \to 0$, the reader can readily check that doing so does not affect the correctness of our next computations.}

The first lemma we prove is the following: it permits to control small-time excursions of the characteristics $\varphi^\epsilon$, in terms of the time increment $\Delta$ and the parameter $\epsilon$.
\begin{lem} \label{lem:phi_eps_incr}
Let $T>0$, $p\geq 1$, $\Delta>0$. Then
\begin{equation*}
\mathbb{E} \left[ 
\sup_{\substack{t+\delta \leq T \\ \delta \leq \Delta}}
\| \varphi^\epsilon_{t+\delta} - \varphi^\epsilon_t \|^p_{L^{\infty}(\mathbb{T}^2,\mathbb{T}^2)}\right] \lesssim
\frac{\Delta^p}{\epsilon^p}.
\end{equation*}
\end{lem}
\begin{proof}
The increment $\varphi^\epsilon_{t+\delta}(x) - \varphi^\epsilon_{t}(x)$ can be written as
\begin{align*}
\varphi^\epsilon_{t+\delta}(x) - \varphi^\epsilon_{t}(x) 
= &\int_{t}^{t+\delta}
u^\epsilon_s(\varphi^\epsilon_s(x)) ds +
\int_{t}^{t+\delta} 
\sum_{k \in \mathbb{N}} \sigma_k(\varphi^\epsilon_s(x))\eta^{\epsilon,k}_s ds,
\end{align*}
therefore, since $\|u^\epsilon_s\|_{L^\infty(\mathbb{T}^2,\mathbb{R}^2)} \lesssim
\|\xi_0\|_{L^\infty(\mathbb{T}^2)}$, we have 
\begin{align*}
\sup_{\substack{t+\delta \leq T}}
\| \varphi^\epsilon_{t+\delta} - \varphi^\epsilon_t \|_{L^{\infty}(\mathbb{T}^2,\mathbb{T}^2)}
\lesssim\,
&\delta\left( 1+\sum_{k \in \mathbb{N}} 
\|\sigma_k\|_{L^\infty(\mathbb{T}^2,\mathbb{R}^2)}
\sup_{s\in[0,T]} |\eta^{\epsilon,k}_s|\right).
\end{align*}
The thesis follows by \autoref{lem:OU}.
\end{proof}

The previous lemma can be slightly improved by the following:
\begin{lem} \label{lem:phi_eps_incr_bis}
For every $T>0$, $p\geq 1$ and fixed $n=0,\dots,T/\delta_\epsilon -1$ we have
\begin{equation*}
\mathbb{E} \left[ 
\| \varphi^\epsilon_{(n+1)\delta_\epsilon} - \varphi^\epsilon_{n\delta_\epsilon} \|^p_{L^{\infty}(\mathbb{T}^2,\mathbb{T}^2)}\right] \lesssim
\frac{\delta_\epsilon^{2p}}{\epsilon^{2p}} + \delta_\epsilon^{p/2} + \epsilon^p.
\end{equation*}
\end{lem}
\begin{proof}
The increment $\varphi^\epsilon_{(n+1)\delta_\epsilon}(x) - \varphi^\epsilon_{n\delta_\epsilon}(x)$ can be written as
\begin{align*}
\varphi^\epsilon_{(n+1)\delta_\epsilon}(x) - \varphi^\epsilon_{n\delta_\epsilon}(x)
= &\int_{n\delta_\epsilon}^{(n+1)\delta_\epsilon}
u^\epsilon_s(\varphi^\epsilon_s(x)) ds \\
&+
\int_{n\delta_\epsilon}^{(n+1)\delta_\epsilon}
\sum_{k \in \mathbb{N}} \left( \sigma_k(\varphi^\epsilon_s(x))-\sigma_k(\varphi^\epsilon_{n\delta_\epsilon}(x)) \right)\eta^{\epsilon,k}_s ds \\
&+
\int_{n\delta_\epsilon}^{(n+1)\delta_\epsilon}
\sum_{k \in \mathbb{N}} 
\sigma_k(\varphi^\epsilon_{n\delta_\epsilon}(x)) \eta^{\epsilon,k}_s ds.
\end{align*}
The first term is easy, and can be controlled as in \autoref{lem:phi_eps_incr}. The second one is bounded in $L^\infty(\mathbb{T}^2,\mathbb{T}^2)$ uniformly in $n$ by
\begin{align*}
\int_{0}^{\delta_\epsilon}
\sum_{k \in \mathbb{N}} 
\|\nabla \sigma_k \|_{L^\infty(\mathbb{T}^2,\mathbb{R}^4)}
\sup_{\substack{t+s \leq T}}
\| \varphi^\epsilon_{t+s} - \varphi^\epsilon_t \|_{L^{\infty}(\mathbb{T}^2,\mathbb{T}^2)}
\sup_{s\in[0,T]} |\eta^{\epsilon,k}_s| ds,
\end{align*}
and by H\"older inequality with exponent $q>1$
\begin{align*}
\mathbb{E}& \left[ \left( \int_{0}^{\delta_\epsilon}
\sum_{k \in \mathbb{N}} 
\|\nabla \sigma_k \|_{L^\infty(\mathbb{T}^2,\mathbb{R}^4)}
\sup_{\substack{t+s \leq T}}
\| \varphi^\epsilon_{t+s} - \varphi^\epsilon_t \|_{L^{\infty}(\mathbb{T}^2,\mathbb{T}^2)}
\sup_{s\in[0,T]} |\eta^{\epsilon,k}_s| ds
\right)^p \right] \\
&\leq
\delta_\epsilon^{p-1}
\left( \sum_{k \in \mathbb{N}} 
\|\nabla \sigma_k \|_{L^\infty(\mathbb{T}^2,\mathbb{R}^4)} \right)^{p-1} 
\int_{0}^{\delta_\epsilon}
\sum_{k \in \mathbb{N}} 
\|\nabla \sigma_k \|_{L^\infty(\mathbb{T}^2,\mathbb{R}^4)}\\
&\quad \times
\mathbb{E} \left[ \sup_{\substack{t+s \leq T}}
\| \varphi^\epsilon_{t+s} - \varphi^\epsilon_t \|^{pq}_{L^{\infty}(\mathbb{T}^2,\mathbb{T}^2)} \right]^{1/q} 
\mathbb{E} \left[ \sup_{s\in[0,T]} |\eta^{\epsilon,k}_s|^{pq'}  \right]^{1/q'} ds \\
&\lesssim
\delta_\epsilon^{p-1}
\int_{0}^{\delta_\epsilon}
\frac{s^p}{\epsilon^{2p}} ds
\lesssim
\frac{\delta_\epsilon^{2p}}{\epsilon^{2p}}.
\end{align*}
The third term is bounded in $L^\infty(\mathbb{T}^2,\mathbb{R}^2)$ by
\begin{align*}
\sum_{k \in \mathbb{N}}
\| \sigma_k \|_{L^\infty(\mathbb{T}^2,\mathbb{R}^2)}
\left| 
\int_{n\delta_\epsilon}^{(n+1)\delta_\epsilon}
\eta^{\epsilon,k}_s ds \right|
=
\sum_{k \in \mathbb{N}}
\| \sigma_k \|_{L^\infty(\mathbb{T}^2,\mathbb{R}^2)}
\left| 
\beta^{\epsilon,k}_{(n+1)\delta_\epsilon}-\beta^{\epsilon,k}_{n\delta_\epsilon} \right|,
\end{align*}
where $\beta^{\epsilon,k}_t$ stands for the integrated Ornstein-Uhlenbeck process:
\begin{align*}
\beta^{\epsilon,k}_t = \int_0^t \eta^{\epsilon,k}_s ds, 
\quad
t \in [0,T], k \in \mathbb{N}.
\end{align*}
Using 
\begin{align*}
\beta^{\epsilon,k}_{(n+1)\delta_\epsilon}-\beta^{\epsilon,k}_{n\delta_\epsilon} 
=
\beta^{k}_{(n+1)\delta_\epsilon}-\beta^{k}_{n\delta_\epsilon} 
- \epsilon^2
\left( 
\eta^{\epsilon,k}_{(n+1)\delta_\epsilon}-\eta^{\epsilon,k}_{n\delta_\epsilon} \right)
\end{align*}
and \autoref{lem:OU} we get
\begin{align*}
\mathbb{E}& \left[ \left( \sum_{k \in \mathbb{N}}
\| \sigma_k \|_{L^\infty(\mathbb{T}^2,\mathbb{R}^2)}
\left| 
\beta^{\epsilon,k}_{(n+1)\delta_\epsilon}-\beta^{\epsilon,k}_{n\delta_\epsilon} \right|
\right)^p \right] 
\lesssim
\delta_\epsilon^{p/2} + \epsilon^p.
\end{align*}
\end{proof}
Next, we move to the analogous estimate for the limiting characteristics $\varphi$. 
Denote by $c:\mathbb{T}^2 \to \mathbb{R}^2$ the following Stratonovich corrector:
\begin{align*}
c(x)=\frac12 \sum_{k \in \mathbb{N}} \nabla\sigma_k(x) \cdot \sigma_k(x),
\quad
x \in \mathbb{T}^2,
\end{align*}
which allows to rewrite \eqref{eq:char} in the following It\=o form:
\begin{align*}
	\begin{cases}
d\varphi_{t}(x) =
u_t\left(  \varphi_{t}(x)\right)  dt +
c\left(  \varphi_{t}(x)\right)  dt +
\sum_{k}\sigma_{k}\left(  \varphi_{t}(x)\right)
 d\beta_{t}^{k},\\ 
\varphi_{0}\left(  x\right) = x.
	\end{cases}
\end{align*}

\begin{lem} \label{lem:phi_incr}
Let $T>0$, $p\geq 1$, $\Delta>0$. Then for every fixed $n=0,\dots,T/\delta_\epsilon -1$:
\begin{equation*}
\mathbb{E} \left[ 
\sup_{\substack{\delta \leq \Delta}}
\| \varphi_{n\delta_\epsilon+\delta} - \varphi_{n\delta_\epsilon} \|^p_{L^{\infty}(\mathbb{T}^2,\mathbb{T}^2)}\right] \lesssim 
\Delta^{p/2}.
\end{equation*}
\end{lem}
\begin{proof}
The increment $\varphi_{n\delta_\epsilon+\delta}(x) - \varphi_{n\delta_\epsilon}(x)$ can be written as
\begin{align*}
\varphi_{n\delta_\epsilon+\delta}(x) - \varphi_{n\delta_\epsilon}(x) 
= &\int_{n\delta_\epsilon}^{n\delta_\epsilon+\delta}
u_s(\varphi_s(x)) ds +
\int_{n\delta_\epsilon}^{n\delta_\epsilon+\delta}
c(\varphi_s(x))ds \\
&+
\int_{n\delta_\epsilon}^{n\delta_\epsilon+\delta} 
\sum_{k \in \mathbb{N}} \sigma_k(\varphi_s(x))d\beta^{k}_s.
\end{align*}
The first two terms are easy and can be handled as usual. On the other hand, using Burkholder-Davis-Gundy inequality, for fixed $n$ and $k$ the last term is controlled by
\begin{align*}
\mathbb{E} \left[ 
\sup_{\delta \leq \Delta}
\left\| \int_{n\delta_\epsilon}^{n\delta_\epsilon+\delta} 
\sigma_k(\varphi_s(x))d\beta^{k}_s \right\|^p_{L^{\infty}(\mathbb{T}^2,\mathbb{R}^2)}\right]
\lesssim
\Delta^{p/2} \|\sigma_k\|^p_{L^{\infty}(\mathbb{T}^2,\mathbb{R}^2)},
\end{align*}
hence, for fixed $n$ and for $\alpha=1-1/p$, H\"older inequality with exponent $p$ gives
\begin{align*}
\mathbb{E}& \left[ 
\sup_{\delta \leq \Delta}
\left\| 
\sum_{k \in \mathbb{N}}
\int_{n\delta_\epsilon}^{n\delta_\epsilon+\delta}
\sigma_k(\varphi_s(x))d\beta^{k}_s \right\|^p_{L^{\infty}(\mathbb{T}^2,\mathbb{R}^2)}\right] \\
&=
\mathbb{E} \left[ 
\sup_{\delta \leq \Delta}
\left\| 
\sum_{k \in \mathbb{N}}
\|\sigma_k\|^{\alpha}_{L^{\infty}(\mathbb{T}^2,\mathbb{R}^2)}
\|\sigma_k\|^{-\alpha}_{L^{\infty}(\mathbb{T}^2,\mathbb{R}^2)}
\int_{n\delta_\epsilon}^{n\delta_\epsilon+\delta} 
\sigma_k(\varphi_s(x))d\beta^{k}_s \right\|^p_{L^{\infty}(\mathbb{T}^2,\mathbb{R}^2)}\right] \\
&\leq
\left( \sum_{k \in \mathbb{N}}
\|\sigma_k\|_{L^{\infty}(\mathbb{T}^2,\mathbb{R}^2)}\right)^{p}\Delta^{p/2} .
\end{align*}
\end{proof}

\subsection{The Nakao method} \label{ssec:nakao}
The argument presented in this paragraph is due to Nakao and can be found, for instance, in \cite{IkWa14}. Roughly speaking, it allows to exploit the discretization of the equation to show the closeness, in a certain sense to be specified, between the Stratonovich corrector and the iterated integral of the Ornstein-Uhlenbeck process.

First we need some preparation.
For any $n=0,\dots,T/\delta_\epsilon-1$, consider the following decomposition:
\begin{align*}
\int_{n\delta_\epsilon}^{(n+1)\delta_\epsilon} &\sum_{k \in \mathbb{N}} \sigma_k(\varphi^\epsilon_s(x)) \eta^{\epsilon,k}_s ds \\
=
&\int_{n\delta_\epsilon}^{(n+1)\delta_\epsilon} \sum_{k \in \mathbb{N}} \left(\sigma_k(\varphi^\epsilon_s(x))-\sigma_k(\varphi^\epsilon_{n\delta_\epsilon}(x))\right)
\eta^{\epsilon,k}_s ds \\
&+
\int_{n\delta_\epsilon}^{(n+1)\delta_\epsilon} \sum_{k \in \mathbb{N}} \sigma_k(\varphi^\epsilon_{n\delta_\epsilon}(x))
\eta^{\epsilon,k}_s ds \\
=
&
\int_{n\delta_\epsilon}^{(n+1)\delta_\epsilon}
\sum_{k \in \mathbb{N}}\left(
\int_{n\delta_\epsilon}^{s} 
\nabla\sigma_k(\varphi^\epsilon_r(x)) \cdot
u^\epsilon_r(\varphi^\epsilon_r(x)) dr
\right)\eta^{\epsilon,k}_s ds \\
&+
\int_{n\delta_\epsilon}^{(n+1)\delta_\epsilon}
\sum_{k,h \in \mathbb{N}}\left(
\int_{n\delta_\epsilon}^{s} 
\nabla\sigma_k(\varphi^\epsilon_r(x)) \cdot
\sigma_h(\varphi^\epsilon_r(x)) \eta^{\epsilon,h}_r dr
\right)\eta^{\epsilon,k}_s ds \\
&+
\int_{n\delta_\epsilon}^{(n+1)\delta_\epsilon} \sum_{k \in \mathbb{N}} \sigma_k(\varphi^\epsilon_{n\delta_\epsilon}(x))
d\beta^k_s \\
&-
\int_{n\delta_\epsilon}^{(n+1)\delta_\epsilon} \sum_{k \in \mathbb{N}} \sigma_k(\varphi^\epsilon_{n\delta_\epsilon}(x))
\epsilon^2 d\eta^{\epsilon,k}_s  \\
=&\, I^\epsilon_1(n) + I^\epsilon_2(n) + I^\epsilon_3(n) + I^\epsilon_4(n).
\end{align*}

We further decompose
\begin{align*}
I^\epsilon_2(n) =
&\int_{n\delta_\epsilon}^{(n+1)\delta_\epsilon}
\sum_{k,h \in \mathbb{N}} \bigg(
\int_{n\delta_\epsilon}^{s} 
\Big( \nabla\sigma_k(\varphi^\epsilon_r(x)) \cdot
\sigma_h(\varphi^\epsilon_r(x)) \\
& \quad
-\nabla\sigma_k(\varphi^\epsilon_{n\delta_\epsilon}(x)) \cdot
\sigma_h(\varphi^\epsilon_{n\delta_\epsilon}(x)) \Big)\eta^{\epsilon,h}_r dr
\bigg)\eta^{\epsilon,k}_s ds \\
&+
\int_{n\delta_\epsilon}^{(n+1)\delta_\epsilon}
\sum_{k,h \in \mathbb{N}} \bigg(
\int_{n\delta_\epsilon}^{s} 
\Big( \nabla\sigma_k(\varphi^\epsilon_{n\delta_\epsilon}(x)) \cdot
\sigma_h(\varphi^\epsilon_{n\delta_\epsilon}(x)) \\ 
& \quad
-\nabla\sigma_k(\varphi_{n\delta_\epsilon}(x)) \cdot
\sigma_h(\varphi_{n\delta_\epsilon}(x))\Big)\eta^{\epsilon,h}_r dr
\bigg)\eta^{\epsilon,k}_s ds \\
&+
\int_{n\delta_\epsilon}^{(n+1)\delta_\epsilon}
\sum_{k,h \in \mathbb{N}} \left(
\int_{n\delta_\epsilon}^{s} 
\nabla\sigma_k(\varphi_{n\delta_\epsilon}(x)) \cdot
\sigma_h(\varphi_{n\delta_\epsilon}(x))\eta^{\epsilon,h}_r dr
\right)\eta^{\epsilon,k}_s ds \\
=&\,
I^\epsilon_{2a}(n)+I^\epsilon_{2b}(n)+I^\epsilon_{2c}(n).
\end{align*}

Regarding the limiting Stratonovich integral, we can rewrite:
\begin{align*}
\int_{n\delta_\epsilon}^{(n+1)\delta_\epsilon}\sum_{k\in \mathbb{N}} \sigma_k(\varphi_s(x)) \circ d\beta^k_s 
=
&\int_{n\delta_\epsilon}^{(n+1)\delta_\epsilon}
\sum_{k\in \mathbb{N}} \left(\sigma_k(\varphi_s(x))-\sigma_k(\varphi_{n\delta_\epsilon}(x))\right) d\beta^k_s \\
&+
\int_{n\delta_\epsilon}^{(n+1)\delta_\epsilon}
\sum_{k\in \mathbb{N}} \sigma_k(\varphi_{n\delta_\epsilon}(x)) 
d\beta^k_s \\
&+
\int_{n\delta_\epsilon}^{(n+1)\delta_\epsilon}
\left( c(\varphi_s(x))- c(\varphi_{n\delta_\epsilon}(x))  \right) ds \\
&+
\int_{n\delta_\epsilon}^{(n+1)\delta_\epsilon}
c(\varphi_{n\delta_\epsilon}(x)) ds \\
=&\, J^\epsilon_1(n) + J^\epsilon_2(n) + J^\epsilon_3(n) + J^\epsilon_4(n).
\end{align*}

\begin{lem} \label{lem:est_aux}
The following inequalities hold:
\begin{align*}
\mathbb{E} \left[ \sup_{m=1,\dots,T/\delta_\epsilon} \left\| 
\sum_{n=0}^{m-1} I^\epsilon_1(n) \right\|_{L^1(\mathbb{T}^2,\mathbb{R}^2)}\right]
&\lesssim \frac{\delta_\epsilon}{\epsilon};\\
\mathbb{E} \left[ \sup_{m=1,\dots,T/\delta_\epsilon} \left\| 
\sum_{n=0}^{m-1} I^\epsilon_{2a}(n) \right\|_{L^1(\mathbb{T}^2,\mathbb{R}^2)}\right]
&\lesssim \frac{\delta_\epsilon^2}{\epsilon^3} ;\\
\mathbb{E} \left[ \sup_{m=1,\dots,T/\delta_\epsilon} \left\| 
\sum_{n=0}^{m-1} I^\epsilon_4(n) \right\|_{L^1(\mathbb{T}^2,\mathbb{R}^2)}\right]
&\lesssim \frac{\delta_\epsilon}{\epsilon} +
\frac{\epsilon}{\delta_\epsilon^{1/2}} +
\frac{\epsilon^2}{\delta_\epsilon} +
\epsilon;\\
\mathbb{E} \left[ \sup_{m=1,\dots,T/\delta_\epsilon} \left\| 
\sum_{n=0}^{m-1} J^\epsilon_1(n) \right\|_{L^1(\mathbb{T}^2,\mathbb{R}^2)}\right]
&\lesssim \delta_\epsilon^{1/2} ;\\
\mathbb{E} \left[ \sup_{m=1,\dots,T/\delta_\epsilon} \left\| 
\sum_{n=0}^{m-1} J^\epsilon_3(n) \right\|_{L^1(\mathbb{T}^2,\mathbb{R}^2)}\right]
&\lesssim \delta_\epsilon^{1/2}.\\
\end{align*}
In particular, all the quantities above go to zero as $\epsilon \to 0$, under the condition $\delta_\epsilon^2/\epsilon^3 \to 0$, $\delta_\epsilon/\epsilon^2 \to \infty$.
\end{lem}
\begin{proof}
Consider first $I^\epsilon_1(n)$.
Using $\| u^\epsilon _r\|_{L^\infty(\mathbb{T}^2,\mathbb{R}^2)} \lesssim \| \xi_0 \|_{L^\infty(\mathbb{T}^2)}$ for every $r \in [0,T]$, we get 
\begin{align*} 
\mathbb{E}& \left[ \sup_{m=1,\dots,T/\delta_\epsilon} \left\| 
\sum_{n=0}^{m-1} I^\epsilon_1(n) \right\|_{L^1(\mathbb{T}^2,\mathbb{R}^2)}\right] \\
&\lesssim
\mathbb{E} \left[
\sum_{n=0}^{T/\delta_\epsilon-1}
\int_{n\delta_\epsilon}^{(n+1)\delta_\epsilon}
\sum_{k \in \mathbb{N}}
\| \nabla \sigma_k \|_{L^\infty(\mathbb{T}^2,\mathbb{R}^4)}
(s-n\delta_\epsilon) |\eta^{\epsilon,k}_s| ds
\right] \\
&\lesssim
\sum_{k \in \mathbb{N}}
\| \nabla \sigma_k \|_{L^\infty(\mathbb{T}^2,\mathbb{R}^4)}
\delta_\epsilon
\mathbb{E} \left[ \sup_{s \in [0,T]}|\eta^{\epsilon,k}_s| \right] 
\lesssim \frac{\delta_\epsilon}{\epsilon}.
\end{align*}
For the term $I^\epsilon_{2a}(n)$, \autoref{lem:phi_eps_incr} gives:
\begin{align*} 
\mathbb{E}& \left[ \sup_{m=1,\dots,T/\delta_\epsilon} \left\| 
\sum_{n=0}^{m-1} I^\epsilon_{2a}(n) \right\|_{L^1(\mathbb{T}^2,\mathbb{R}^2)}\right] \\
&\leq
\sum_{k,h \in \mathbb{N}}
\left( 
\|\nabla^2 \sigma_k \|_{L^\infty(\mathbb{T}^2,\mathbb{R}^8)}
\|\sigma_h \|_{L^\infty(\mathbb{T}^2,\mathbb{R}^2)} +
\|\nabla \sigma_k \|_{L^\infty(\mathbb{T}^2,\mathbb{R}^4)}
\|\nabla \sigma_h \|_{L^\infty(\mathbb{T}^2,\mathbb{R}^4)} 
\right) \\
&\quad \times
\mathbb{E} \left[
\sum_{n=0}^{T/\delta_\epsilon-1}
\int_{n\delta_\epsilon}^{(n+1)\delta_\epsilon}
\int_{n\delta_\epsilon}^{s}
\| \varphi^\epsilon_r - \varphi^\epsilon_{n\delta_\epsilon}\|_{L^\infty(\mathbb{T}^2,\mathbb{T}^2)}
|\eta^{\epsilon,h}_r||\eta^{\epsilon,k}_s| dr ds
\right] \\
&\lesssim
\sum_{k,h \in \mathbb{N}}
\left( 
\|\nabla^2 \sigma_k \|_{L^\infty(\mathbb{T}^2,\mathbb{R}^8)}
\|\sigma_h \|_{L^\infty(\mathbb{T}^2,\mathbb{R}^2)} +
\|\nabla \sigma_k \|_{L^\infty(\mathbb{T}^2,\mathbb{R}^4)}
\|\nabla \sigma_h \|_{L^\infty(\mathbb{T}^2,\mathbb{R}^4)} 
\right) \\
&\quad \times
\sum_{n=0}^{T/\delta_\epsilon-1}
\int_{n\delta_\epsilon}^{(n+1)\delta_\epsilon}
\int_{n\delta_\epsilon}^{s}
\frac{r-n\delta_\epsilon}{\epsilon^3} dr ds
\lesssim 
\frac{\delta_\epsilon^2}{\epsilon^3}.
\end{align*}
The term $I^\epsilon_{4}(n)$ is treated after a discrete integration by parts, in order to have a better control of the time increment: indeed, \autoref{lem:phi_eps_incr_bis} gives
\begin{align*} 
\mathbb{E}& \left[ \sup_{m=1,\dots,T/\delta_\epsilon} \left\| 
\sum_{n=0}^{m-1} I^\epsilon_{4}(n) \right\|_{L^1(\mathbb{T}^2,\mathbb{R}^2)}\right] \\
&\leq
\mathbb{E} \left[
\sup_{m=1,\dots,T/\delta_\epsilon} \left\| 
\sum_{n=0}^{m-1} 
\sum_{k \in \mathbb{N}}
\sigma_k(\varphi^\epsilon_{n\delta_\epsilon}(x))
\epsilon^2
\left( \eta^{\epsilon,k}_{(n+1)\delta_\epsilon}
-\eta^{\epsilon,k}_{n\delta_\epsilon} \right) \right\|_{L^1(\mathbb{T}^2,\mathbb{R}^2)}\right] \\
&\lesssim
\mathbb{E} \left[
\sup_{m=1,\dots,T/\delta_\epsilon} \left\| 
\sum_{n=1}^{m} 
\sum_{k \in \mathbb{N}}
\left( 
\sigma_k(\varphi^\epsilon_{n\delta_\epsilon}(x)) - 
\sigma_k(\varphi^\epsilon_{(n-1)\delta_\epsilon}(x)) \right)
\epsilon^2
\eta^{\epsilon,k}_{n\delta_\epsilon} \right\|_{L^1(\mathbb{T}^2,\mathbb{R}^2)}\right] \\
&\quad +
\mathbb{E} \left[
\left\|  
\sum_{k \in \mathbb{N}}
\sigma_k(\varphi^\epsilon_{0}(x))
\epsilon^2
\eta^{\epsilon,k}_{0} \right\|_{L^1(\mathbb{T}^2,\mathbb{R}^2)}
\right] \\
&\quad +
\mathbb{E} \left[
\sup_{m=1,\dots,T/\delta_\epsilon}
\left\|  
\sum_{k \in \mathbb{N}}
\sigma_k(\varphi^\epsilon_{m}(x))
\epsilon^2
\eta^{\epsilon,k}_{m} \right\|_{L^1(\mathbb{T}^2,\mathbb{R}^2)}
\right]  \\
&\lesssim
\mathbb{E} \left[ 
\sum_{n=1}^{T/\delta_\epsilon} 
\sum_{k \in \mathbb{N}}
\|\nabla \sigma_k \|_{L^\infty(\mathbb{T}^2,\mathbb{R}^4)}
\left\| 
\varphi^\epsilon_{n\delta_\epsilon} - 
\varphi^\epsilon_{(n-1)\delta_\epsilon}\right\|
_{L^\infty(\mathbb{T}^2,\mathbb{T}^2)}
\epsilon^2
|\eta^{\epsilon,k}_{n\delta_\epsilon}| \right]  + \epsilon \\
&\lesssim
\sum_{n=1}^{T/\delta_\epsilon} 
\sum_{k \in \mathbb{N}}
\|\nabla \sigma_k \|_{L^\infty(\mathbb{T}^2,\mathbb{R}^4)}
\left( \frac{\delta_\epsilon^{2}}{\epsilon^{2}} + \delta_\epsilon^{1/2} + \epsilon\right) \epsilon + \epsilon \\
&\lesssim
\frac{\delta_\epsilon}{\epsilon} +
\frac{\epsilon}{\delta_\epsilon^{1/2}} +
\frac{\epsilon^2}{\delta_\epsilon} +
\epsilon.
\end{align*}
For the remaining terms $J^\epsilon_1(n)$ and $J^\epsilon_3(n)$, we have by \autoref{lem:phi_incr}
\begin{align*} 
\mathbb{E}& \left[ \sup_{m=1,\dots,T/\delta_\epsilon} \left\| 
\sum_{n=0}^{m-1} J^\epsilon_1(n) \right\|_{L^1(\mathbb{T}^2,\mathbb{R}^2)}\right] \\
&\lesssim
\sum_{k \in \mathbb{N}}
\| \nabla \sigma_k \|_{L^\infty(\mathbb{T}^2,\mathbb{R}^4)}
\mathbb{E} \left[ \left(
\sum_{n=0}^{T/\delta_\epsilon-1}
\int_{n\delta_\epsilon}^{(n+1)\delta_\epsilon}
\| \varphi_s - \varphi_{n\delta_\epsilon}\|^2_{L^\infty(\mathbb{T}^2,\mathbb{T}^2)} ds \right)^{1/2} \right]\\
&\lesssim
\sum_{k \in \mathbb{N}}
\| \nabla \sigma_k \|_{L^\infty(\mathbb{T}^2,\mathbb{R}^4)}
\mathbb{E} \left[ 
\sum_{n=0}^{T/\delta_\epsilon-1}
\int_{n\delta_\epsilon}^{(n+1)\delta_\epsilon}
\| \varphi_s - \varphi_{n\delta_\epsilon}\|^2_{L^\infty(\mathbb{T}^2,\mathbb{T}^2)} ds  \right]^{1/2} 
\lesssim \delta_\epsilon^{1/2},
\end{align*}
and similarly
\begin{align*} 
\mathbb{E}& \left[ \sup_{m=1,\dots,T/\delta_\epsilon} \left\| 
\sum_{n=0}^{m-1} J^\epsilon_3(n) \right\|_{L^1(\mathbb{T}^2,\mathbb{R}^2)}\right] \\
&\leq
\sum_{k\in \mathbb{N}}
\left( 
\|\nabla^2 \sigma_k \|_{L^\infty(\mathbb{T}^2,\mathbb{R}^8)}
\|\sigma_k \|_{L^\infty(\mathbb{T}^2,\mathbb{R}^2)} +
\|\nabla \sigma_k \|_{L^\infty(\mathbb{T}^2,\mathbb{R}^4)}
\|\nabla \sigma_k \|_{L^\infty(\mathbb{T}^2,\mathbb{R}^4)} 
\right) \\
&\quad \times
\mathbb{E} \left[
\sum_{n=0}^{T/\delta_\epsilon-1}
\int_{n\delta_\epsilon}^{(n+1)\delta_\epsilon}
\| \varphi_s - \varphi_{n\delta_\epsilon}\|_{L^\infty(\mathbb{T}^2,\mathbb{T}^2)} ds
\right] 
\lesssim \delta_\epsilon^{1/2}.
\end{align*}
\end{proof}

\begin{lem}[Nakao] \label{lem:nakao}
The following inequality holds:
\begin{align*}
\mathbb{E} \left[ \sup_{m=1,\dots,T/\delta_\epsilon} \left\| 
\sum_{n=0}^{m-1} I^\epsilon_{2c}(n) - J^\epsilon_4(n) \right\|_{L^1(\mathbb{T}^2,\mathbb{R}^2)}\right]
&\lesssim 
\frac{\delta_\epsilon}{\epsilon}
+ \delta_\epsilon^{1/2}
+\frac{\epsilon^2}{\delta_\epsilon}.
\end{align*}
\end{lem}
\begin{proof}
By the very definition of $I^\epsilon_{2c}(n)$, $J^\epsilon_4(n)$, one has
\begin{align*}
I^\epsilon_{2c}(n)
=&
\sum_{k,h \in \mathbb{N}} 
\nabla\sigma_k(\varphi_{n\delta_\epsilon}(x)) \cdot
\sigma_h(\varphi_{n\delta_\epsilon}(x)) 
\int_{n\delta_\epsilon}^{(n+1)\delta_\epsilon}\left(
\int_{n\delta_\epsilon}^{s} 
\eta^{\epsilon,h}_r dr
\right)\eta^{\epsilon,k}_s ds, \\
J^\epsilon_4(n) 
=&
\sum_{k \in \mathbb{N}} 
\nabla\sigma_k(\varphi_{n\delta_\epsilon}(x)) \cdot
\sigma_k(\varphi_{n\delta_\epsilon}(x))\, 
\frac{\delta_\epsilon}{2}.
\end{align*}
Therefore, one can decompose the quantity under investigation as follows:
\begin{align}\label{eq:nakao_dec}
\sum_{n=0}^{m-1}& I^\epsilon_{2c}(n) - J^\epsilon_4(n) \\
=&  \nonumber
\sum_{n=0}^{m-1} 
\sum_{k,h \in \mathbb{N}} 
\nabla\sigma_k(\varphi_{n\delta_\epsilon}(x)) \cdot
\sigma_h(\varphi_{n\delta_\epsilon}(x))
\left( c_{h,k}^n(\delta_\epsilon,\epsilon) -\mathbb{E} \left[ c_{h,k}^n(\delta_\epsilon,\epsilon) \mid \mathcal{F}_{n\delta_\epsilon}\right]  \right) \\
&+  \nonumber
\sum_{n=0}^{m-1} 
\sum_{k,h \in \mathbb{N}} 
\nabla\sigma_k(\varphi_{n\delta_\epsilon}(x)) \cdot
\sigma_h(\varphi_{n\delta_\epsilon}(x))
\left( 
\mathbb{E} \left[ c_{h,k}^n(\delta_\epsilon,\epsilon) \mid \mathcal{F}_{n\delta_\epsilon}\right] - \frac{\delta_{h,k}}{2} \delta_\epsilon
 \right),
\end{align} 
where $\delta_{h,k}$ is the Kronecker delta function and
\begin{equation*}
c_{h,k}^n(\delta_\epsilon,\epsilon) 
= 
\int_{n\delta_\epsilon}^{(n+1)\delta_\epsilon}\left(
\int_{n\delta_\epsilon}^{s} 
\eta^{\epsilon,h}_r dr
\right)\eta^{\epsilon,k}_s ds.
\end{equation*}
Notice that $c_{h,k}^n(\delta_\epsilon,\epsilon) $ is measurable with respect to $\mathcal{F}_{(n+1)\delta_\epsilon}$ and has conditional expectation
\begin{align*}
\mathbb{E}&\left[ c_{h,k}^n(\delta_\epsilon,\epsilon) \mid \mathcal{F}_{n\delta_\epsilon}\right] \\ 
&= 
\int_{n\delta_\epsilon}^{(n+1)\delta_\epsilon}\left(
\int_{n\delta_\epsilon}^{s} 
\mathbb{E} \left[ 
\eta^{\epsilon,h}_r \eta^{\epsilon,k}_s \mid \mathcal{F}_{n\delta_\epsilon}\right]  dr
\right)\eta^{\epsilon,k}_s ds
\\
&= \eta^{\epsilon,h}_{n\delta_\epsilon} \eta^{\epsilon,k}_{n\delta_\epsilon} 
\int_{n\delta_\epsilon}^{(n+1)\delta_\epsilon}\left(
\int_{n\delta_\epsilon}^{s} 
e^{-\epsilon^{-2}(r+s-2n\delta_\epsilon)}   dr
\right) ds 
\\
&\quad+ \delta_{h,k}
\int_{n\delta_\epsilon}^{(n+1)\delta_\epsilon}\left(
\int_{n\delta_\epsilon}^{s} 
\frac{\varepsilon^{-2}}{2} 
\left( e^{-\varepsilon^{-2}(s-r)} - e^{-\varepsilon^{-2}(r+s-2n\delta_\epsilon)} \right) dr
\right) ds,
\end{align*}
where we have used the mild formulation of $\eta^{\epsilon}$:
\begin{align*}
\eta^{\epsilon,h}_r = 
e^{-\epsilon^{-2}(r-n\delta_\epsilon)} \eta^{\epsilon,h}_{n\delta_\epsilon} + 
\int_{n\delta_\epsilon}^{r}
\epsilon^{-2} e^{-\epsilon^{-2}(r-r')}  d\beta^{h}_{r'} , \\
\eta^{\epsilon,k}_s = 
e^{-\epsilon^{-2}(s-n\delta_\epsilon)} \eta^{\epsilon,k}_{n\delta_\epsilon} + 
\int_{n\delta_\epsilon}^{s}
\epsilon^{-2} e^{-\epsilon^{-2}(s-s')}  d\beta^{k}_{s'}.
\end{align*}
An elementary computation gives:
\begin{align} \label{eq:cond_exp}
\mathbb{E} \left[ c_{h,k}^n(\delta_\epsilon,\epsilon)  \mid \mathcal{F}_{n\Delta}\right] 
= 
&\frac{\varepsilon^4}{2} 
\eta^{\epsilon,h}_{n\delta_\epsilon} \eta^{\epsilon,k}_{n\delta_\epsilon}
\left( e^{-\epsilon^{-2}\delta_\epsilon} -1 \right)^2 \\
&+ \nonumber
\frac{\delta_{h,k}}{2} \left( \delta_\epsilon + \epsilon^2 \left( - \frac{3}{2} + 2 e^{-\epsilon^{-2}\delta_\epsilon} - \frac{1}{2} e^{-2\epsilon^{-2}\delta_\epsilon}\right)  \right).
\end{align}

Since the quantity
\begin{align*}
M_m(x) =
\sum_{n=0}^{m-1} 
\sum_{k,h \in \mathbb{N}} 
\nabla\sigma_k(\varphi_{n\delta_\epsilon}(x)) \cdot
\sigma_h(\varphi_{n\delta_\epsilon}(x))
\left( c_{h,k}^n(\delta_\epsilon,\epsilon) -\mathbb{E} \left[ c_{h,k}^n(\delta_\epsilon,\epsilon) \mid \mathcal{F}_{n\delta_\epsilon}\right]  \right) 
\end{align*}
is a $L^2(\mathbb{T}^2,\mathbb{R}^2)$-valued martingale with respect to the filtration $(\mathcal{F}_{n\delta_\epsilon})_{n \in \mathbb{N}}$ (crf. \cite{Pi16}), by Doob maximal inequality and martingale property we have the following:
\begin{align*}
\mathbb{E} \left[ \sup_{m=1,\dots,T/\delta_\epsilon} \left\| 
M_m \right\|^2_{L^2(\mathbb{T}^2,\mathbb{R}^2)}\right]
&\lesssim
\mathbb{E} \left[ \left\| 
M_{T/\delta_\epsilon} \right\|^2_{L^2(\mathbb{T}^2,\mathbb{R}^2)}\right] \\
&\lesssim
\mathbb{E} \left[ 
\sum_{n=0}^{T/\delta_\epsilon-1} \left\| 
M_{m+1}-M_m \right\|^2_{L^2(\mathbb{T}^2,\mathbb{R}^2)}\right].
\end{align*}
The conditional expectation is a $L^2(\Omega)$-projection, thus for every $n,h,k \in \mathbb{N}$
\begin{align*}
\mathbb{E} \left[ \left|
c_{h,k}^n(\delta_\epsilon,\epsilon) -\mathbb{E} \left[ c_{h,k}^n(\delta_\epsilon,\epsilon) \mid \mathcal{F}_{n\delta_\epsilon}\right] \right|^2 \right]
&\lesssim
\mathbb{E} \left[ \left| 
c_{h,k}^n(\delta_\epsilon,\epsilon) \right|^2 \right],
\end{align*}
and therefore
\begin{align*}
\mathbb{E}& \left[ 
\sum_{n=0}^{T/\delta_\epsilon-1} \left\| M_{m+1}-M_m \right\|^2_{L^2(\mathbb{T}^2,\mathbb{R}^2)}\right] \\
&\lesssim
\sum_{n=0}^{T/\delta_\epsilon-1} 
\left( 
\sum_{h,k \in \mathbb{N}}
\|\nabla \sigma_k\|_{L^\infty(\mathbb{T}^2,\mathbb{R}^4)}\|
\| \sigma_h\|_{L^\infty(\mathbb{T}^2,\mathbb{R}^2)}
\right) \\
&\quad \times 
\left( 
\sum_{h,k \in \mathbb{N}}
\|\nabla \sigma_k\|_{L^\infty(\mathbb{T}^2,\mathbb{R}^4)}\|
\| \sigma_h\|_{L^\infty(\mathbb{T}^2,\mathbb{R}^2)}
\mathbb{E} \left[ \left| 
c_{h,k}^n(\delta_\epsilon,\epsilon) \right|^2 \right]\right) 
\lesssim 
\frac{\delta_\epsilon^2}{\epsilon^2} + \delta_\epsilon.
\end{align*}

Moreover, the process
\begin{align*}
N_m(x)
=
\sum_{n=0}^{m-1} 
\sum_{k,h \in \mathbb{N}} 
\nabla\sigma_k(\varphi_{n\delta_\epsilon}(x)) \cdot
\sigma_h(\varphi_{n\delta_\epsilon}(x))
\left( 
\mathbb{E} \left[ c_{h,k}^n(\delta_\epsilon,\epsilon) \mid \mathcal{F}_{n\delta_\epsilon}\right] - \frac{\delta_{h,k}}{2} \delta_\epsilon
 \right)
\end{align*}
satisfies
\begin{align*}
\mathbb{E}& \left[ \sup_{m=1,\dots,T/\delta_\epsilon} \left\| 
N_m \right\|^2_{L^2(\mathbb{T}^2,\mathbb{R}^2)}\right]
\lesssim
\frac{\epsilon^4}{\delta_\epsilon^2},
\end{align*}
which is an easy consequence of \eqref{eq:cond_exp}.
By \eqref{eq:nakao_dec} and H\"older inequality, we get
\begin{align*}
\mathbb{E}& \left[ \sup_{m=1,\dots,T/\delta_\epsilon} \left\| 
\sum_{n=0}^{m-1} I^\epsilon_{2c}(n) - J^\epsilon_4(n) \right\|_{L^1(\mathbb{T}^2,\mathbb{R}^2)}\right] \\
&\lesssim 
\mathbb{E}\left[ \sup_{m=1,\dots,T/\delta_\epsilon} \left\| 
M_m \right\|_{L^1(\mathbb{T}^2,\mathbb{R}^2)}\right]
+
\mathbb{E}\left[ \sup_{m=1,\dots,T/\delta_\epsilon} \left\| 
N_m \right\|_{L^1(\mathbb{T}^2,\mathbb{R}^2)}\right] \\
&\lesssim 
\mathbb{E}\left[ \sup_{m=1,\dots,T/\delta_\epsilon} \left\| 
M_m \right\|^2_{L^2(\mathbb{T}^2,\mathbb{R}^2)}\right]^{1/2}
+
\mathbb{E}\left[ \sup_{m=1,\dots,T/\delta_\epsilon} \left\| 
N_m \right\|^2_{L^2(\mathbb{T}^2,\mathbb{R}^2)}\right]^{1/2} \\
&\lesssim 
\frac{\delta_\epsilon}{\epsilon}
+ \delta_\epsilon^{1/2}
+\frac{\epsilon^2}{\delta_\epsilon}.
\end{align*}
\end{proof}

We conclude this paragraph with the following result.
\begin{lem}
The following estimates hold:
\begin{align*}
\mathbb{E}& \left[ \sup_{m=1,\dots,N} \left\| 
\sum_{n=0}^{m-1} I^\epsilon_{3}(n) - J^\epsilon_2(n) \right\|_{L^1(\mathbb{T}^2,\mathbb{R}^2)}\right] \\
&\quad\lesssim 
\sum_{m=1}^N \delta_\epsilon
\mathbb{E} \left[ \sup_{n=1,\dots,m} \left\| 
\varphi^\epsilon_{n\delta_\epsilon}-
\varphi_{n\delta_\epsilon} \right\|_{L^1(\mathbb{T}^2,\mathbb{T}^2)}\right] ;\\
\mathbb{E}& \left[ \sup_{m=1,\dots,N} \left\| 
\sum_{n=0}^{m-1} I^\epsilon_{2b}(n) \right\|_{L^1(\mathbb{T}^2,\mathbb{R}^2)}\right] \\
&\quad\lesssim 
\sum_{m=1}^N \delta_\epsilon
\mathbb{E} \left[ \sup_{n=1,\dots,m} \left\| 
\varphi^\epsilon_{n\delta_\epsilon}-
\varphi_{n\delta_\epsilon} \right\|_{L^1(\mathbb{T}^2,\mathbb{T}^2)}\right] .
\end{align*}
\end{lem}
\begin{proof}
The first estimate is an easy consequence of Burkholder-Davis-Gundy inequality.
For the second estimate, one can argue as in \autoref{lem:nakao} to replace the quantity $I^\epsilon_{2b}(n)$ with:
\begin{align*}
\int_{n\delta_\epsilon}^{(n+1)\delta_\epsilon}
\left( c(\varphi^\epsilon_{n\delta_\epsilon}) -
c(\varphi_{n\delta_\epsilon})\right) ds
=
\delta_\epsilon
\left( c(\varphi^\epsilon_{n\delta_\epsilon}) -
c(\varphi_{n\delta_\epsilon})\right),
\end{align*}
up to a correction that is infinitesimal as $\epsilon \to 0$. For the latter quantity, the desired inequality is immediate.
\end{proof}

\subsection{Proof of \autoref{prop:char}}

We are ready to prove the main result of this section. Recall:
\begin{equation*}
d\varphi_{t}^{\epsilon}\left(  x\right)  = u^{\epsilon}_t\left(
\varphi_{t}^{\epsilon}\left(  x\right)  \right)dt  + \sum_{k \in \mathbb{N}} \sigma_k \left( \varphi_{t}^{\epsilon}\left(  x\right)  \right) \eta^{\epsilon,k}_t dt.
\end{equation*}
Since $\varphi^\epsilon : \mathbb{T}^2 \to \mathbb{T}^2$ is measure-preserving, for Lebesgue a.e. $x \in \mathbb{T}^2$: 
\begin{equation*}
u^\epsilon_t (\varphi^\epsilon_t(x))
=
\int_{\mathbb{T}^2} K(\varphi^\epsilon_t(x) - \varphi^\epsilon_t(y)) \xi_0(y) dy,
\end{equation*}
and therefore we have the following integral formulation for \eqref{eq:char_eps}
\begin{equation*}
\varphi^\epsilon_t(x) = x +
\int_0^t \left( \int_{\mathbb{T}^2} K(\varphi^\epsilon_s(x) - \varphi^\epsilon_s(y)) \xi_0(y) dy\right) ds +
\int_0^t \sum_{k \in \mathbb{N}} \sigma_k (\varphi^\epsilon_s(x)) \eta^{\epsilon,k}_s ds,
\end{equation*}
and similarly for \eqref{eq:char}
\begin{equation*}
\varphi_t(x) = x +
\int_0^t \left( \int_{\mathbb{T}^2} K(\varphi_s(x) - \varphi_s(y)) \xi_0(y) dy\right) ds +
\int_0^t \sum_{k \in \mathbb{N}} \sigma_k (\varphi_s(x)) \circ d\beta^k_s.
\end{equation*}

\begin{proof}[Proof of \autoref{prop:char}]

For the difference $Z^\epsilon_t(x) = \varphi^\epsilon_t(x)-\varphi_t(x)$ we have:
\begin{align*}
Z^\epsilon_t(x)
=&
\int_0^t \left( \int_{\mathbb{T}^2} 
\left(
K(\varphi^\epsilon_s(x) - \varphi^\epsilon_s(y)) -
K(\varphi^\epsilon_s(x) - \varphi_s(y)) \right)  \xi_0(y) dy\right) ds \\
&+
\int_0^t \left( \int_{\mathbb{T}^2} 
\left(
K(\varphi^\epsilon_s(x) - \varphi_s(y)) -
K(\varphi_s(x) - \varphi_s(y)) \right)\xi_0(y) dy\right) ds \\
&+
\int_0^t \sum_{k\in \mathbb{N}} 
\sigma_k(\varphi^\epsilon_s(x)) \eta^{\epsilon,k}_s ds -
\int_0^t \sum_{k\in \mathbb{N}} 
\sigma_k(\varphi_s(x)) \circ d\beta^k_s.
\end{align*}
Using the estimates given by \autoref{lem:phi_eps_incr} and  \autoref{lem:phi_incr}, we can approximate the latter two integrals in the expression above with their discretized versions, computed in a point $n\delta_\epsilon$ such that $t \in [n\delta_\epsilon,(n+1)\delta_\epsilon)$, up to a correction that is infinitesimal as $\epsilon \to 0$.
Then, in the regime $\delta_\epsilon^2/\epsilon^3 \to 0$, $\delta_\epsilon/\epsilon^2 \to \infty$, using the results of \autoref{ssec:nakao}, \autoref{lem:log_lip} and the concavity of $\gamma$ we arrive to
\begin{align*}
\mathbb{E} \left[ \sup_{s \leq t} \left\| 
Z^\epsilon_{s} \right\|_{L^1(\mathbb{T}^2,\mathbb{T}^2)}\right]
&\lesssim 
\int_0^t \gamma
\left( \mathbb{E} \left[ \sup_{r \leq s} \left\| 
Z^\epsilon_{r} \right\|_{L^1(\mathbb{T}^2,\mathbb{T}^2)}\right]\right) ds \\
&+
r^\epsilon_T
+
\sum_{n=1}^{ \lfloor t/\delta_\epsilon \rfloor} \delta_\epsilon
\mathbb{E} \left[ \sup_{r \leq n} \left\| 
Z^\epsilon_{r} \right\|_{L^1(\mathbb{T}^2,\mathbb{T}^2)}\right]\\
&\lesssim 
\int_0^t \gamma
\left( \mathbb{E} \left[ \sup_{r \leq s} \left\| 
Z^\epsilon_{r} \right\|_{L^1(\mathbb{T}^2,\mathbb{T}^2)}\right]\right) ds +
r^\epsilon_T,
\end{align*}
{ where $r^\epsilon_T$ is a remainder coming from the discretization procedure, \autoref{lem:est_aux} and \autoref{lem:nakao},} and it goes to zero as $\epsilon \to 0$.
By \autoref{lem:comp}, we conclude that 
\begin{align*}
\mathbb{E} \left[ \sup_{s \leq T} \left\| 
Z^\epsilon_{s} \right\|_{L^1(\mathbb{T}^2,\mathbb{T}^2)}\right] \to 0,
\end{align*}
and therefore $Z^\epsilon \to 0$ in mean value as a variable in $C([0,T],L^1(\mathbb{T}^2,\mathbb{T}^2))$.
\end{proof}

\autoref{lem:log_lip} and the same calculations as above yield the following convergence at the velocity level:
\begin{cor} \label{cor:vel}
Assume (A1), and let $u^{\epsilon} = K \ast \xi^{\epsilon}$ (resp. $u= K \ast \xi$) be the velocity field associated with the characteristics $\varphi^\epsilon$ (resp. $\varphi$).
Then as $\epsilon \to 0$:
\begin{align*}
\mathbb{E} \left[ \sup_{s \leq T} \left\| 
u^\epsilon_{s}-u_s \right\|_{L^1(\mathbb{T}^2,\mathbb{R}^2)}\right]
&\lesssim
\gamma \left( 
\mathbb{E} \left[ \sup_{s \leq T} \left\| 
\varphi^\epsilon_{s}-\varphi_s \right\|_{L^1(\mathbb{T}^2,\mathbb{T}^2)}\right] \right) \to 0.
\end{align*}
\end{cor}

\section{Convergence of the vorticity process} \label{sec:conv_vort}

In this brief section, we discuss the consequences of the convergence of the characteristics at the level of the vorticity process.
Recall that we have proved:
\begin{align*}
\mathbb{E} \left[ \sup_{s \leq T} \left\| 
\varphi^\epsilon_{s}-\varphi_s \right\|_{L^1(\mathbb{T}^2,\mathbb{T}^2)}\right] \to 0,
\end{align*}
as $\epsilon \to 0$.
We have the following:
\begin{thm} \label{thm:vort}
The vorticity process $\xi^\epsilon$ solution of the simplified system \eqref{eq:xi_eps} converges to $\xi$ solution of \eqref{eq:xi} as $\epsilon \to 0$ in the following sense:
for every $f \in L^1(\mathbb{T}^2)$:
\begin{align*}
\mathbb{E} \left[\left|
\int_{\mathbb{T}^2}
\xi_t^{\epsilon}(x)  f\left(  x\right)  dx
-
\int_{\mathbb{T}^2}
\xi_t(x)  f\left(  x\right)  dx
\right|\right] \to 0
\end{align*}
as $\epsilon \to 0$, for every fixed $t \in [0,T]$ and in $L^p([0,T])$ for every finite $p$.
\end{thm}

\begin{proof}
By \eqref{eq:transport_eps} and the fact that $\varphi^\epsilon : \mathbb{T}^2 \to \mathbb{T}^2$ is measure-preserving, a change of variable leads to
\begin{align*}
\int_{\mathbb{T}^2}
\xi_t^{\epsilon}(x)  f\left(  x\right)  dx
&=
\int_{\mathbb{T}^2}
\xi_{0}\left(
y\right)  f\left(  \varphi_{t}^{\epsilon}\left(  y\right)  \right)  dy,
\end{align*}
for every $t \in [0,T]$ and $f \in L^1(\mathbb{T}^2)$. Similarly, 
\begin{align*}
\int_{\mathbb{T}^2}
\xi_t(x)  f\left(  x\right)  dx
&=
\int_{\mathbb{T}^2}
\xi_{0}\left(  y\right)
f\left(  \varphi_{t}\left(  y\right)  \right)  dy.
\end{align*}

Since $f \in L^1(\mathbb{T}^2)$, then by Lusin theorem \cite[Theorem 2.23]{Ru70} for every $\delta>0$ there exists a continuous function $f_\delta \in C(\mathbb{T}^2)$ and a compact set $C_\delta$ such that $f$ concides with $f_\delta$ on $C_\delta$ and $meas(\mathbb{T}^2 \setminus C_\delta) < \delta$. Therefore
\begin{align*}
&\left| 
\int_{\mathbb{T}^2}
\xi_{0}\left(
y\right)  f\left(  \varphi_{t}^{\epsilon}\left(  y\right)  \right)  dy-
\int_{\mathbb{T}^2}
\xi_{0}\left(  y\right)
f\left(  \varphi_{t}\left(  y\right)  \right)  dy\right| \\
&\qquad \leq
\| \xi_0 \|_{L^\infty(\mathbb{T}^2)}
\int_{C_\delta}
\left| 
f(\varphi_{t}^{\epsilon}(y))-
f(\varphi_{t}(y)) \right| dy \\
&\qquad +
\| \xi_0 \|_{L^\infty(\mathbb{T}^2)}
\int_{\mathbb{T}^2 \setminus C_\delta}
\left| f(\varphi_{t}^{\epsilon}(y))\right| dy\\
&\qquad +
\| \xi_0 \|_{L^\infty(\mathbb{T}^2)}
\int_{\mathbb{T}^2 \setminus C_\delta}
\left| f(\varphi_{t}(y))\right| dy.
\end{align*}
Since $|f| \in L^1(\mathbb{T}^2)$ and $\varphi^\epsilon_t$, $\varphi_t$ are measure-preserving, absolute continuity of Lebesgue integral gives: for every $\delta'>0$ there exists $\delta>0$ such that, for every $\epsilon>0$, $t \in [0,T]$ and a.e. $\omega \in \Omega$:
\begin{align*}
\int_{\mathbb{T}^2 \setminus C_\delta}
\left| f(\varphi_{t}^{\epsilon}(y))\right| dy
+
\int_{\mathbb{T}^2 \setminus C_\delta}
\left| f(\varphi_{t}(y))\right| dy
< \delta'.
\end{align*}
It remains to study the quantity
\begin{align*}
\int_{C_\delta}
\left| 
f(\varphi_{t}^{\epsilon}(y))-
f(\varphi_{t}(y)) \right| dy
\leq
\int_{\mathbb{T}^2}
\left| 
f_\delta(\varphi_{t}^{\epsilon}(y))-
f_\delta(\varphi_{t}(y)) \right| dy.
\end{align*}
Since $f_\delta$ is continuous, one can argue as in \cite[Proposition 6.2]{BrFlMa16} to get
\begin{align*}
\mathbb{E} \left[\int_{\mathbb{T}^2}
\left| 
f_\delta(\varphi_{t}^{\epsilon}(y))-
f_\delta(\varphi_{t}(y)) \right| dy
\right] \to 0
\end{align*}
as $\epsilon \to 0$, for every fixed $t \in [0,T]$ and in $L^p([0,T])$ for every finite $p$. Putting all together, the proof is complete.
\end{proof}

\section{Back to 2D Euler equations} \label{sec:full}
In this section we focus back to the \textit{full} 2D Euler system \eqref{eq:full_euler}
\begin{align*}
	\begin{cases}
d\xi_{\text{L}}^{\epsilon}+
u_{\text{L}}^{\epsilon}\cdot\nabla\xi_{\text{L} 
}^{\epsilon} dt  = 
-u_{\text{S}}^{\epsilon}\cdot\nabla\xi_{\text{L} 
}^{\epsilon} dt,\\
d\xi_{\text{S}}^{\epsilon}+
u_{\text{L}}^{\epsilon}\cdot\nabla\xi_{\text{S}
}^{\epsilon}dt   =-\epsilon^{-2}
\xi_{\text{S}}^{\epsilon}dt+\epsilon^{-2}dW_t, \\
\xi_{\text{L}}^{\epsilon}|_{t=0}   =\xi_{0},\quad\xi_{\text{S}}^{\epsilon
}|_{t=0}=\xi_{\text{S}}^{0,\epsilon}. \nonumber
	\end{cases}
\end{align*}

Recall that the Brownian motion $W(t,x)$ is given by
\begin{align*}
W(t,x) =
\sum_{k \in \mathbb{N}}
\theta_k(x) \beta^k_t,
\end{align*}
where the coefficients $\theta_k$ satisfy assumption (A1):
\begin{gather*}
\theta_k \in L^2_0(\mathbb{T}^2) \cap C^1(\mathbb{T}^2,\mathbb{R}), \\
\sum_{k \in \mathbb{N}}
\| \nabla \theta_k \|_{L^\infty(\mathbb{T}^2,\mathbb{R}^2)} 
< \infty.
\end{gather*}

Let $\Theta^\epsilon(t,x) = \sum_{k \in \mathbb{N}} \theta_k(x) \eta^{\epsilon,k}_t$ be the Ornstein-Uhlenbeck process solution of
\begin{align*}
d\Theta^\epsilon
=
-\epsilon^{-2} \Theta^\epsilon dt
+\epsilon^{-2}dW_t,
\end{align*}
with initial condition $\xi_{\text{S}}^{0,\epsilon}$. For simplicity, { take $\xi_{\text{S}}^{0,\epsilon}$ as in \autoref{sec:not}, so that $\Theta^\epsilon$ is a stationary process, progressively measurable with respect to the filtration $\{\mathcal{F}_t\}_{t \geq 0}$, and $\eta^{\epsilon,k}$ is independent of $\eta^{\epsilon,h}$ for $k \neq h$}.
Notice that the regularity of $\Theta^\epsilon$ is the same of $W$. In particular, under assumption (A1), $\Theta^\epsilon$ takes a.s. values in $C([0,T],W^{1,\infty}(\mathbb{T}^2))$ and
\begin{align} \label{eq:Theta}
\mathbb{E} \left[
\sup_{t \in [0,T]}
\| \nabla \Theta^\epsilon \|
_{L^\infty(\mathbb{T}^2,\mathbb{R}^2)}
\right] \leq C \epsilon^{-1}.
\end{align}

Define the difference process:
\begin{align*}
\zeta^\epsilon = \xi_{\text{S}}^{\epsilon} - \Theta^\epsilon, 
\end{align*}
which solves the equation
\begin{align*}
d\zeta^\epsilon+
u_{\text{L}}^{\epsilon}\cdot\nabla\zeta^\epsilon dt  =& 
-\epsilon^{-2}
\zeta^\epsilon dt
-u_{\text{L}}^{\epsilon}\cdot\nabla \Theta^\epsilon dt, \\
\zeta^\epsilon_0 =&\, 0.
\end{align*}

We consider also the following auxiliary process
\begin{align*}
\tilde{\zeta}^\epsilon_t 
= 
e^{\epsilon^{-2}t}\zeta^\epsilon _t,
\end{align*} 
which solves the same equation without damping:
\begin{align*}
d\tilde{\zeta}^\epsilon+
u_{\text{L}}^{\epsilon}\cdot\nabla\tilde{\zeta}^\epsilon dt  =& 
-u_{\text{L}}^{\epsilon}\cdot\nabla \Theta^\epsilon dt, \\
\tilde{\zeta}^\epsilon_0 =& 0.
\end{align*}

By \cite[Theorem 6.1.6]{Ku97} the problem above admits formally an unique solution:
\begin{align} \label{eq:kunita}
\tilde{\zeta}^\epsilon(t,x)
=
-\int_0^t
(u_{\text{L}}^{\epsilon}\cdot\nabla \Theta^\epsilon)(s,(\phi^\epsilon_{s,t})^{-1}(x))ds,
\end{align}
where $\phi^\epsilon$ is defined by
\begin{align*}
d \phi^\epsilon_{s,t}(x) 
=
u_{\text{L}}^{\epsilon}(t,\phi^\epsilon_{s,t}(x)) dt, 
\quad
\phi^\epsilon_{s,s}(x) = x.
\end{align*}
Recalling the equality $\tilde{\zeta}^\epsilon_t 
= 
e^{\epsilon^{-2}t}\zeta^\epsilon _t$, the equation above becomes
\begin{align*}
{\zeta}^\epsilon(t,x)
=
-e^{-\epsilon^{-2}t}
\int_0^t
(u_{\text{L}}^{\epsilon}\cdot\nabla \Theta^\epsilon)(s,\phi_{s,t}^{-1}(x))ds.
\end{align*}

By difference, we recover the small scale vorticity $\xi_{\text{S}}^{\epsilon}$
and we can plug it into the equation for the large scale vorticity $\xi_{\text{L}}^{\epsilon}$ to obtain:
\begin{align} \label{eq:large}
d\xi_{\text{L}}^{\epsilon}+
u_{\text{L}}^{\epsilon}\cdot\nabla\xi_{\text{L} %
}^{\epsilon} dt
= &
-(K\ast \Theta^\epsilon)
\cdot\nabla\xi_{\text{L} }^{\epsilon} dt \\
&+e^{-\epsilon^{-2}t} \nonumber
K \ast \left( \int_0^t
(u_{\text{L}}^{\epsilon}\cdot\nabla \Theta^\epsilon)(s,\phi_{s,t}^{-1}(\cdot))ds \right)
\cdot\nabla\xi_{\text{L} }^{\epsilon} dt
\end{align}

This is a (highly non-linear) transport equation with random coefficients.
It is worth noticing that we have obtained equation \eqref{eq:large} above by a formal application of \eqref{eq:kunita}, somewhat in the same spirit of formal integration-by-parts performed when dealing with weak solutions of certain PDEs.
Well posedness, in the Lagrangian sense -- that is, the analogous of \autoref{def:sol} -- of \eqref{eq:large} is the content of the following:

\begin{prop} \label{prop:well_posedness}
For every $\epsilon >0$, equation \eqref{eq:large} admits a unique weakly progressively measurable solution $\xi^\epsilon_L$, given by the transportation of the initial vorticity $\xi_0 $ along the characteristics $\psi^\epsilon$ defined below.
Moreover, the characteristics $\psi^\epsilon$ of the full equation \eqref{eq:large} converge to the characteristics  \eqref{eq:char} as $\epsilon \to 0$ in the following sense:
\begin{align*}
\mathbb{E} \left[ \sup_{s \leq T} \left\| 
\psi^\epsilon_{s}-\varphi_s \right\|_{L^1(\mathbb{T}^2,\mathbb{T}^2)}\right] \to 0.
\end{align*}
\end{prop}

\begin{proof}

Consider the following system of characteristics:
\begin{align} \label{eq:char_full} \tag{C} 
	\left\{	
	\begin{array}{rl}
d \psi^\epsilon_{t}(x) 
=&
u_{\text{L}}^{\epsilon}(t,\psi^\epsilon_{t}(x)) dt +
(K\ast \Theta^\epsilon) (t,\psi^\epsilon_{t}(x)) dt \\ &+
(K\ast \zeta^\epsilon) (t,\psi^\epsilon_{t}(x)) dt,
\quad
\psi^\epsilon_{0}(x)= x, \vspace{0.1cm}\\ 
\xi_{\text{L}}^{\epsilon}(t,x)
=&
\xi_0(({\psi^\epsilon}_t)^{-1}(x)), \vspace{0.1cm}\\
u_{\text{L}}^{\epsilon}(t,x)
=&
(K \ast \xi_{\text{L}}^{\epsilon})(t,x), \vspace{0.1cm}\\
d \phi^\epsilon_{s,t}(x) 
=& 
u_{\text{L}}^{\epsilon}(t,\phi^\epsilon_{s,t}(x)) dt, 
\quad
\phi^\epsilon_{s,s}(x) = x,\vspace{0.1cm}\\
{\zeta}^\epsilon (t,x)
=&
-e^{-\epsilon^{-2}t}
\int_0^t
(u_{\text{L}}^{\epsilon}\cdot\nabla \Theta^\epsilon)(s,(\phi_{s,t}^\epsilon)^{-1}(x))ds.	
	\end{array}
	\right.
\end{align}

Notice that the only unknown of the system above is the characteristic $\psi^\epsilon$, the other quantities being uniquely determined $\omega$-wise by the former, see \cite[Section 3]{BrFlMa16}.
We prove path-by-path well posedness of the system \eqref{eq:char_full} in the class of flows of measure-preserving homeomorphisms.
The argument is similar to that of the proof of \cite[Theorem 3.4]{BrFlMa16}.

Let $\mathcal{M}_T$ be the space
\begin{align*}
\mathcal{M}_T 
= 
\Big\{
&\psi :  [0,T] \times \mathbb{T}^2 \to \mathbb{T}^2 \mbox{ measurable }, 
\\
&\psi_t \mbox{  is a measure-preserving homeomorphism for every }
t \in [0,T]
\Big\}.
\end{align*}
The proof relies on a Picard iteration with \emph{fixed} $\omega$. Let $G:\mathcal{M}_T \mapsto \mathcal{M}_T$ be the map
that at every $\psi$ associates the solution $G(\psi)$ of the equation: 
\begin{align*}
G(\psi)_t (x)
=&
\,x 
+
\int_0^t
u(s,G(\psi)_s(x)) ds
\\
&+
\int_0^t
(K \ast \Theta^\epsilon)(s,G(\psi)_s(x)) ds \\
&+
\int_0^t
(K \ast \zeta)(s,G(\psi)_s(x)) ds,
\end{align*}
where $u=u^\psi$, $\zeta=\zeta^\psi$ are computed from $\psi$ and not from $G(\psi)$, so that the $G(\psi)$ is well defined for every $\psi \in  \mathcal{M}_T$ (see \cite[Lemma 3.1]{BrFlMa16} for the analogous result for Euler equations).
For any $\psi, \psi' \in \mathcal{M}_T$, a computation similar to that of the proof of \autoref{prop:char} yields the key inequality
\begin{align*}
\left\|
G(\psi)_t - G(\psi')_t \right\|_{L^1(\mathbb{T}^2,\mathbb{T}^2)}
\lesssim& \left( 1+ \sup_{s \leq T}
\left\|
\nabla \Theta^\epsilon_s \right\|_{L^\infty(\mathbb{T}^2,\mathbb{R}^2)}
\right) \\
&\quad \times
\int_0^t
\gamma \left( 
\left\|
G(\psi)_s - G(\psi')_s \right\|_{L^1(\mathbb{T}^2,\mathbb{T}^2)}
\right)
ds \\
&+
 \left( 1+ \sup_{s \leq T}
\left\|
\nabla \Theta^\epsilon_s \right\|_{L^\infty(\mathbb{T}^2,\mathbb{R}^2)}
\right) \\
&\quad \times
\int_0^t
\gamma \left( 
\left\|
\psi_s - \psi'_s \right\|_{L^1(\mathbb{T}^2,\mathbb{T}^2)}
\right)
ds,
\end{align*}
which guarantees the a.s. convergence of the Picard iteration towards a solution of the system \eqref{eq:char_full} on the time interval $[0,T_1]$, where $0<T_1\leq T$ may depend on $\omega$.
However, for any fixed $\omega$, one can iterate this procedure with the same time step $T_1$, to obtain existence on $[0,T]$ after $N=N(\omega)$ iterations of the argument.
In addition, having care to initialize the iteration scheme with a $\mathcal{F}_0$ measurable random element of $\mathcal{M}_T$ (take for instance $\psi^0_t(x)=x$ for every $t$) we also obtain progressively measurability of the solution so constructed with respect to the filtration $\{\mathcal{F}_t\}_{t \in [0,T]}$.
Uniqueness is obtained applying the same Picard scheme to two solutions $\psi, \psi' \in \mathcal{M}_T$.
We omit the remaining details, which are contained in \cite{BrFlMa16}.

Let us now investigate the convergence of characteristics $\psi^\epsilon \to \varphi$.
The proof is the same as \autoref{prop:char}. We do not repeat it here, and we limit ourselves to notice that, since $\sup_{t \in [0,T]}
\| \xi^\epsilon_L (t)\|_{L^\infty(\mathbb{T}^2)}
\leq
\| \xi_0 \|_{L^\infty(\mathbb{T}^2)}$, we have
\begin{align*}
\sup_{s \leq T} \left\| 
\zeta^\epsilon_s \right\|_{L^1(\mathbb{T}^2)}
\lesssim 
\| \xi_0 \|_{L^\infty(\mathbb{T}^2)}
\epsilon^2 \sup_{s \leq T}
\left\|
\nabla \Theta^\epsilon_s \right\|_{L^\infty(\mathbb{T}^2,\mathbb{R}^2)}.
\end{align*}
By \eqref{eq:Theta}, the expected value of this quantity is infinitesimal as $\epsilon \to 0$ and therefore does not affect the argument of \autoref{prop:char}. 
\end{proof}

In virtue of the previous proposition, we deduce the analogous of \autoref{thm:vort} and \autoref{cor:vel} for the \emph{full} 2D Euler system \eqref{eq:full_euler}, that is \autoref{thm:intro} in the Introduction, whose precise formulation is the following:

\begin{thm} \label{thm:conv_vort_full}
Assume (A1), and let $\xi_{\text{L}}^{\epsilon}$ be the large scale process solution of \eqref{eq:full_euler} in the sense of \autoref{prop:well_posedness}, 
$\xi_{\text{L}}$ be the solution of \eqref{eq:xi}.
Then $\xi_{\text{L}}^{\epsilon}$ converges as $\epsilon\rightarrow0$ to $\xi_{\text{L}}$ in the following sense: 
for every $f \in L^1(\mathbb{T}^2)$:
\begin{align*}
\mathbb{E} \left[\left|
\int_{\mathbb{T}^2}
\xi_{\text{L}}^{\epsilon}(t,x)  f\left(  x\right)  dx
-
\int_{\mathbb{T}^2}
\xi_{\text{L}}(t,x)  f\left(  x\right)  dx
\right|\right] \to 0
\end{align*}
as $\epsilon \to 0$, for every fixed $t \in [0,T]$ and in $L^p([0,T])$ for every finite $p$.
Moreover, the large scale velocity process $u_{\text{L}}^{\epsilon} = K \ast \xi_{\text{L}}^{\epsilon}$ converges towards  $u_{\text{L}} = K \ast \xi_{\text{L}}$ as $\epsilon\rightarrow0$, in
mean value, as variables in $C([0,T],L^{1}(\mathbb{T}^{2},\mathbb{R}^2))  $.
\end{thm}

\appendix

\section{Convergence of weak solutions}

Throughout the paper we have used the Lagrangian formulation of Euler equations and, more generally, of transport-type equations.
This point of view turns out to be really effective to investigate the convergence of the large scale component of the system \eqref{eq:full_euler} as $\epsilon \to 0$, since the nonlinear term in the equation of characteristics has been widely studied before.

In this section we aim to link the Lagrangian point of view on Euler equations and other equations of transport-type with the analitically weak point of view.

To fix the ideas, take first 2D Euler equations in vorticity form \eqref{eq:xi_intro}. 
Both Lagrangian formulation and weak formulation aim to give a notion of solution to \eqref{eq:xi_intro} which does not require the regularity needed for classical solution.
In the case of Lagrangian formulation, the space derivative of the solution is formally canceled out by the composition with the characteristics: 
\begin{align*}
\partial_t \xi_t(\varphi_t(x)) = 0.
\end{align*}
In the weak formulation, the derivatives of the solution are formally eliminated by an integration-by-parts formula for the product of the solution against regular test functions: to be precise, a weak solution of \eqref{eq:xi_intro} is given by a function $\xi \in L^\infty([0,T],L^\infty(\mathbb{T}^2))$ such that, for every test function $f \in C^1(\mathbb{T}^2)$, it holds for every $t \in [0,T]$:
\begin{align*}
\int_{\mathbb{T}^2}
\xi_t(x) f(x) dx
=
\int_{\mathbb{T}^2}
\xi_0(x) f(x) dx
+
\int_0^t \left(
\int_{\mathbb{T}^2}
\xi_s(x) (K \ast \xi_s)(x) \cdot \nabla f(x) dx
\right) ds.
\end{align*}

In the case of 2D Euler equations in vorticity form, the Lagrangian point of view and the analitically weak point of view are equivalent, that is, every Lagrangian solution of \eqref{eq:xi_intro} is also a weak solution, and every weak solution of \eqref{eq:xi_intro} is given by the transportation of the initial datum $\xi_0$ along characteristics. The proof of this classical fact under the assumption $\xi_{0} \in L_0^\infty(\mathbb{T}^2)$ can be found in \cite{MaPu94}.

In \cite{BrFlMa16} a similar statement is proved for the limiting process \eqref{eq:xi}:
\begin{align*}
d\xi_t+u_t\cdot\nabla\xi_t dt &=
-\sum_{k \in \mathbb{N}}
\sigma_{k}\cdot\nabla\xi_t\circ d\beta^{k}_t,
\end{align*}
with $u_t=K\ast\xi_t$. In this case, for every initial datum $\xi_{0} \in L_0^\infty(\mathbb{T}^2)$, the Lagrangian formulation of \eqref{eq:xi}, which has been used in the present work, is equivalent to the following distributional formulation (see \cite[Theorem 2.14 and Proposition 5.3]{BrFlMa16}).
\begin{definition}
A weakly progressively measurable process $\xi \in L^\infty(\Omega \times [0,T] \times \mathbb{T}^2 )$ is said to be a \emph{$L^\infty$ distributional solution} to \eqref{eq:xi} if for every test function $f \in C^\infty(\mathbb{T}^2)$ it holds $\mathbb{P}$-a.s.: for every $t \in [0,T]$
\begin{align*}
\int_{\mathbb{T}^2} \xi_t(x) f(x) dx
=&
\int_{\mathbb{T}^2} \xi_0(x) f(x) dx
+
\int_0^t \left(
\int_{\mathbb{T}^2} \xi_s(x) (K\ast\xi_s)(x) \cdot \nabla f(x) dx \right) ds \\
&+ 
\sum_{k \in \mathbb{N}}
\int_0^t \left(
\int_{\mathbb{T}^2} \xi_s(x) \sigma_k(x) \cdot \nabla f(x) dx \right) d\beta^{k}_s \\
&- 
\frac12 \sum_{k \in \mathbb{N}}
\int_0^t \left(
\int_{\mathbb{T}^2} \xi_s(x)
[(\sigma_k(x) \cdot \nabla)\sigma_k(x)]  \cdot \nabla f(x) dx \right) ds \\
&- 
\frac12 \sum_{k \in \mathbb{N}}
\int_0^t \left(
\int_{\mathbb{T}^2} \xi_s(x)
\,\mbox{\emph{div}} [(\sigma_k \cdot \nabla)\sigma_k ](x) f(x) dx \right) ds \\
&+ 
\frac12 \sum_{k \in \mathbb{N}}
\int_0^t \left(
\int_{\mathbb{T}^2} \xi_s(x) 
\,\mbox{\emph{tr}}[\sigma_k \sigma_k^* \nabla^2 f]  (x) dx \right) ds.
\end{align*}
\end{definition}

Following \cite{BrFlMa16}, we define an analogous notion of $L^\infty$ distributional solution to \eqref{eq:large}.
Recall the definition of the Ornstein-Uhlenbeck process:
\begin{align*}
\Theta^\epsilon(t,x) =& \sum_{k \in \mathbb{N}} \theta_k(x) \eta^{\epsilon,k}_t.
\end{align*}

\begin{definition}
A weakly progressively measurable process $\xi^\epsilon \in L^\infty(\Omega \times [0,T] \times \mathbb{T}^2 )$ is said to be a \emph{$L^\infty$ distributional solution} to \eqref{eq:large} if for every test function $f \in C^\infty(\mathbb{T}^2)$ it holds $\mathbb{P}$-a.s.: for every $t \in [0,T]$
\begin{align*}
\int_{\mathbb{T}^2} \xi^\epsilon_t(x) f(x) dx
=&
\int_{\mathbb{T}^2} \xi_0(x) f(x) dx
+
\int_0^t \left(
\int_{\mathbb{T}^2} \xi^\epsilon_s(x) u^{\epsilon}_s(x) \cdot \nabla f(x) dx \right) ds \\
&+
\int_0^t \left(
\int_{\mathbb{T}^2} \xi^\epsilon_s(x) (K\ast \Theta^\epsilon_s)(x) \cdot \nabla f(x) dx \right) ds \\
&+
\int_0^t \left(
\int_{\mathbb{T}^2} \xi^\epsilon_s(x) (K\ast \zeta^\epsilon_s)(x) \cdot \nabla f(x) dx \right) ds, 
\end{align*}
where $u^{\epsilon}(t,x)=(K \ast \xi^{\epsilon})(t,x)$ and $\zeta^\epsilon$ is given by
\begin{align*}
{\zeta}^\epsilon(t,x)
=&
-e^{-\epsilon^{-2}t}
\int_0^t
(u^{\epsilon}\cdot\nabla \Theta^\epsilon)(s,\phi_{s,t}^{-1}(x))ds, \\ 
d \phi^\epsilon_{s,t}(x) 
=& 
u^{\epsilon}(t,\phi^\epsilon_{s,t}(x)) dt, 
\quad
\phi^\epsilon_{s,s}(x) = x.
\end{align*}
\end{definition}

As for the equations \eqref{eq:xi_intro} and \eqref{eq:xi}, the notion of $L^\infty$ distributional solution to \eqref{eq:large} is indeed equivalent to the notion of Lagrangian solution used throughout the paper. This is the content of the forthcoming:

\begin{prop} \label{prop:well_posedness_distr}
The unique Lagrangian solution $\xi^\epsilon$ to \eqref{eq:large} given by \autoref{prop:well_posedness} is also a $L^\infty$ distributional solution.
Conversely, every $L^\infty$ distributional solution to \eqref{eq:large} is also a Lagrangian solution.
\end{prop}

\begin{proof}

\emph{Step 1}. We first prove that the Lagrangian solution to \eqref{eq:large} is also a $L^\infty$ distributional solution. 
Let $\psi^\epsilon$ be the unique stochastic flow of homeomorphism solution of the system of characteristics \eqref{eq:char_full}.
By the representation formula
$\xi^\epsilon_t = \xi_0 \circ (\psi^\epsilon_t)^{-1}$ one immediately has $\xi^\epsilon \in L^\infty(\Omega \times [0,T] \times \mathbb{T}^2)$.
By the a.s. measure-preserving property of $\psi^\epsilon$ one gets, for any $f \in L^1(\mathbb{T}^2)$ :
\begin{align} \label{eq:xi_duality}
\int_{\mathbb{T}^2} \xi^\epsilon_t(x) f(x) dx
=
\int_{\mathbb{T}^2} \xi_0(x) f(\psi^\epsilon_t(x)) dx.
\end{align}
Since for every $x \in \mathbb{T}^2$ the process $t \mapsto \psi^\epsilon_t(x)$ is progressively measurable, one can deduce from \eqref{eq:xi_duality} that $\xi^\epsilon$ is weakly progressively measurable.

Let now $f \in C^\infty(\mathbb{T}^2)$ be a given test function. Define
\begin{align*}
v^\epsilon_s(x)
=
u^\epsilon_s(x)
+
(K\ast \Theta^\epsilon_s)(x)
+
(K\ast \zeta^\epsilon_s)(x).
\end{align*}
Using 
\begin{align*}
d f( \psi^\epsilon_t(x) )
=
v^\epsilon_t(\psi^\epsilon_t(x)) \cdot \nabla f(\psi^\epsilon_t(x) ) dt,
\end{align*}
multiplying per $\xi_0$, integrating in time and space, and using \eqref{eq:xi_duality} one obtains that $\xi^\epsilon$ is a $L^\infty$ distributional solution to \eqref{eq:large}.

\emph{Step 2}. We prove that every $L^\infty$ distributional solution to \eqref{eq:large} is also a Lagrangian solution, \emph{i.e.} it is given by the transportation of the initial vorticity $\xi_0$ along characteristics.

A weakly progressively measurable process $\xi^\epsilon \in L^\infty([0,T] \times \mathbb{T}^2 \times \Omega)$ being a $L^\infty$ distributional solution to \eqref{eq:large} corresponds to: for every test function $f \in C^\infty(\mathbb{T}^2)$ it holds $\mathbb{P}$-a.s.: for every $t \in [0,T]$
\begin{align} \label{eq:distr_form}
\int_{\mathbb{T}^2} \xi^\epsilon_t(x) f(x) dx
=&
\int_{\mathbb{T}^2} \xi_0(x) f(x) dx
+
\int_0^t \left(
\int_{\mathbb{T}^2} \xi^\epsilon_s(x) v^{\epsilon}_s(x) \cdot \nabla f(x) dx \right) ds.
\end{align}
Let $\rho \in C^\infty(\mathbb{R}^2)$ be a non negative even function, supported in $[-1,1]^2$, with $\int_{\mathbb{R}^2} \rho(x)dx=1$. For every $\delta \in (0,1/2)$, denote $\rho_\delta(x)=\delta^{-2} \rho(x/\delta)$ and
\begin{align*}
\xi^{\epsilon,\delta}_t(x) = \int_{\mathbb{R}^2}\rho_\delta (x-y) \xi^\epsilon_t(y)dy.
\end{align*}
In the integral above, the function $\xi^\epsilon_t: \mathbb{T}^2 \to \mathbb{R}$ is interpreted as a periodic function $\xi^\epsilon_t: \mathbb{R}^2 \to \mathbb{R}$ on the full space. The mollified vorticity $\xi^{\epsilon,\delta}_t$ is smooth and periodic in space, therefore the process $\xi^{\epsilon,\delta}$ has a.s. trajectories in 
\begin{align*}
\xi^{\epsilon,\delta} \in L^\infty([0,T],C^\infty(\mathbb{T}^2)).
\end{align*}
Using that $\xi^\epsilon$ is a $L^\infty$ distributional solution to \eqref{eq:large} in the equivalent formulation \eqref{eq:distr_form} with $f=\rho_\delta(x-\cdot)$, one has for every fixed $x \in \mathbb{T}^2$ the a.s. property: for every $t \in [0,T]$ 
\begin{align*}
\xi^{\epsilon,\delta}_t(x)
=
(\xi_0 \ast \rho_\delta)(x)
+
\int_0^t \left(
\int_{\mathbb{T}^2} (\xi^\epsilon_s(y) v^{\epsilon}_s(y) \cdot \nabla 	\rho_\delta(x-y) dy \right) ds.
\end{align*}
Notice that, by a.s. space regularity of the process $\xi^{\epsilon,\delta}$, one can find a full-measure set $\Omega' \subset \Omega$ such that for every $\omega \in \Omega'$ the property above holds simultaneously for every $x \in \mathbb{T}^2$.
Arguing as in \cite{BrFlMa16}, for $\psi^\epsilon$ solving
\begin{align*}
d \psi^\epsilon_{t}(x) 
=
v^{\epsilon}(t,\psi^\epsilon_{t}(x)) dt,
\quad
\psi^\epsilon_{0}(x)= x,
\end{align*}
one can prove a.s. the following: for every $x \in \mathbb{T}^2$ and $t \in [0,T]$
\begin{align} \label{eq:xi_eps_delta}
\xi^{\epsilon,\delta}_t(\psi^\epsilon_{t}(x))
=
(\xi_0 \ast \rho_\delta)(x)
+
\int_0^t
\left[ v^\epsilon_s \cdot \nabla, \ast \rho_\delta\right] \xi^{\epsilon}_s(\psi^\epsilon_{s}(x)) ds.
\end{align}
The commutator above is defined, for fixed divergence-free $v \in W^{1,p}(\mathbb{T}^2,\mathbb{R}^2)$, $p\in[1,\infty)$, and $w \in L^\infty(\mathbb{T}^2)$ by
\begin{align*}
[ v \cdot \nabla, \ast \rho_\delta ] w
=
v \cdot \nabla (\rho_\delta \ast w)
-
\rho_\delta \ast (v \cdot \nabla  w),
\end{align*}
and satisfies (see \cite[Lemma 5.2]{BrFlMa16})
\begin{align} \label{eq:comm_0}
\lim_{\delta \to 0}
[ v \cdot \nabla, \ast \rho_\delta ] w
=& \,
0
\mbox{  in } L^p(\mathbb{T}^2), \\  \label{eq:comm_bound}
\| [ v \cdot \nabla, \ast \rho_\delta ] w \|_{L^p(\mathbb{T}^2)}
\lesssim& \,
\| \nabla v  \|_{L^p(\mathbb{T}^2)}
\| w  \|_{L^\infty(\mathbb{T}^2)}.
\end{align}
By measure-preserving property of $\psi^\epsilon$,
integrating \eqref{eq:xi_eps_delta} in space yields
\begin{align*}
\int_{\mathbb{T}^2} |\xi^{\epsilon,\delta}_t(\psi^\epsilon_{t}(x))
-
(\xi_0 \ast \rho_\delta)(x)| dx
\lesssim
\int_0^t
\int_{\mathbb{T}^2}
|\left[ v^\epsilon_s \cdot \nabla, \ast \rho_\delta\right] \xi^{\epsilon}_s(x)| dx ds.
\end{align*}
Using \eqref{eq:comm_0} and \eqref{eq:comm_bound} with $p=1$,
the bound
\begin{align*}
\| \nabla v^\epsilon_s  \|_{L^1(\mathbb{T}^2)}
\lesssim& \,
\| \xi^\epsilon_s  \|_{L^\infty(\mathbb{T}^2)}
+
\| \Theta^\epsilon_s  \|_{L^\infty(\mathbb{T}^2)}
+
\| \zeta^\epsilon_s  \|_{L^\infty(\mathbb{T}^2)},
\end{align*}
and Lebesgue dominated convergence Theorem, one obtains a.s. the convergence: for every fixed $t \in [0,T]$, as $\delta \to 0$:
\begin{align*}
\xi^{\epsilon,\delta}_t(\psi^\epsilon_{t}(\cdot))
-
(\xi_0 \ast \rho_\delta) 
\to 0
\mbox{  in }
L^1(\mathbb{T}^2).
\end{align*}
By the fact of $\xi^{\epsilon}$ being a.s. with trajectories in $L^\infty(\mathbb{T}^2)$ and $\psi^\epsilon$ being a.s. measure-preserving, one also has a.s. the convergence: for every fixed $t \in [0,T]$, as $\delta \to 0$:
\begin{align*}
\xi^{\epsilon,\delta}_t(\psi^\epsilon_{t}(\cdot))
-
\xi^{\epsilon}_t(\psi^\epsilon_{t}(\cdot))
\to 0
\mbox{  in }
L^1(\mathbb{T}^2),
\end{align*}
and similarly
\begin{align*}
(\xi_0 \ast \rho_\delta) - \xi_0
\to 0
\mbox{  in }
L^1(\mathbb{T}^2).
\end{align*}
We remark that the previous statement asserts the possibility of finding a full-measure set $\Omega'' \subset \Omega$ such that for every $\omega \in \Omega''$ the convergences above hold for every fixed $t \in [0,T]$.
Putting all together, we finally obtain a.s. the  identity: for every $t \in [0,T]$, $\xi^{\epsilon}_t(\psi^\epsilon_{t}(\cdot)) = \xi_0$ as variables in $L^1(\mathbb{T}^2)$, that is for Lebesgue-a.e. $x \in \mathbb{T}^2$. By boundedness, the identity can be understood as variables in $L^\infty(\mathbb{T}^2)$ as well.
\end{proof}

\autoref{prop:well_posedness_distr} above gives well posedness of \eqref{eq:large} in distributional formulation: indeed, the Lagrangian solution given by \autoref{prop:well_posedness} is a $L^\infty$ distributional solution, giving existence; for uniqueness, it suffices to invoke uniqueness of characteristics and the fact that every $L^\infty$ distributional solution is also a Lagrangian solution. 
In terms of convergence, one can therefore restate \autoref{thm:conv_vort_full} with the $L^\infty$ distributional notion of solution.

\end{document}